\documentclass[12pt,a4paper]{article}
\usepackage{amsmath,amssymb,latexsym,amsthm}
\usepackage[english]{babel}
\usepackage[latin2]{inputenc}

\begin{document}

\newtheorem{Dfn}{Definition}
\newtheorem{Theo}{Theorem}
\newtheorem{Lemma}[Theo]{Lemma}
\newtheorem{Prop}[Theo]{Proposition}
\newtheorem{Coro}[Theo]{Corollary}
\newtheorem{Rem}[Theo]{Remark}

\title{Multiplicative loops of $2$-dimensional topological quasifields}
\author{Giovanni Falcone, \'Agota Figula and Karl Strambach}
\date{}
\maketitle

\begin{abstract}
We determine the algebraic structure of the multiplicative loops for locally compact $2$-dimensional topological connected quasifields.
In particular, our attention turns to multiplicative loops which have either a normal subloop of positive dimension or which contain a $1$-dimensional compact subgroup. In the last section we determine explicitly the quasifields which coordinatize locally compact translation planes of  dimension $4$ admitting an at least $7$-dimensional Lie group as collineation group.
\end{abstract}



\bigskip
\centerline{\bf 1. Introduction}

\bigskip
\noindent
Locally compact connected topological non-desarguesian translation planes have been a popular subject of geometrical research since the seventies of the last century (\cite{salzmann}, \cite{betten1}-\cite{betten7}, \cite{knarr}, \cite{ortleb}). These planes are coordinatized by locally compact quasifields $Q$ such that the kernel of $Q$ is either the field $\mathbb R$ of real numbers or the field
$\mathbb C$ of complex numbers (cf. \cite{grundhoefer}, IX.5.5 Theorem, p. 323). If the quasifield $Q$ is $2$-dimensional, then its kernel is
$\mathbb R$.

The classification of  topological  translation planes $\mathcal{A}$ was accomplished by reconstructing the spreads  corresponding to $\mathcal{A}$ from the translation complement which is the stabilizer of a point in the collineation group of $\mathcal{A}$. In this way all planes $\mathcal{A}$ having an at least $7$-dimensional collineation group have been determined (\cite{betten}-\cite{betten6}, \cite{ortleb}).

Although any spread gives the lines through the origin and hence the multiplication in a $2$-dimensional quasifield $Q$ coordinatizing the plane $\mathcal{A}$, to the algebraic structure of the multiplicative loop $Q^{\ast }$ of a proper quasifield $Q$  is not given special attention apart from the facts that the group topologically generated by the left translations of $Q^{\ast }$ is the connected component  of
${\mathrm{GL}}_2(\mathbb R)$, the group topologically generated by the right translations of $Q^{\ast }$ is an infinite-dimensional Lie group
(cf. \cite{loops}, Section 29, p. 345) and any locally compact $2$-dimensional semifield is the field of complex numbers 
(\cite{strambachplaumann}).

Since in the meantime some progress in the classification of compact differentiable loops on the $1$-sphere has been achieved (cf. \cite{figulastrambach}), we believe that loops could have more space in the research concerning $4$-dimensional translation planes.
Using the images of differentiable sections $\sigma :G/H \to G$,  where ${\textstyle H=\left\{ \left( \begin{array}{cc}
a & b \\
0 & a^{-1} \end{array} \right), a> 0, b \in \mathbb R \right\}}$, we classify the $C^1$-differentiable multiplicative loops $Q^{\ast }$ of $2$-dimensional locally compact quasifields $Q$ by functions, the Fourier series of which are described in \cite{figulastrambach}.

The multiplicative loops $Q^{\ast }$ of $2$-dimensional locally compact left quasifields $Q$ for which the set of the left translations of  $Q^{\ast }$ is the product ${\mathcal T} {\mathcal K}$ with $|{\mathcal T} \cap {\mathcal K}| \le 2$, where ${\mathcal T}$ is the set of the left translations of a $1$-dimensional compact loop and ${\mathcal K}$ is the set of the left translations of $Q^{\ast }$ corresponding to the kernel $K_r$ of $Q$, form an important subclass of loops, that we call decomposable loops.
Namely, if $Q^{\ast }$ has a normal subloop of positive dimension or if it contains the group ${\mathrm{SO}}_2(\mathbb R)$, then $Q^{\ast }$ is decomposable. Moreover, we show that any $1$-dimensional $C^1$-differentiable compact loop is a factor of a decomposable  multiplicative loop of a locally compact connected quasifield coordinatizing a $4$-dimensional translation plane. A $2$-dimensional locally compact quasifield $Q$ is the field of complex numbers if and only if the multiplicative loop $Q^{\ast }$ contains a $1$-dimensional normal compact subloop.

Till now mainly those simple loops have been  studied for which the  group generated by their left translations is a simple group. If the group generated by the left translations of a loop $L$ is simple, then $L$ is also simple
(cf. Lemma 1.7 in \cite{loops}).
The multiplicative loops  $Q^{\ast }$ of $2$-dimensional locally compact quasifields show that there are many interesting $2$-dimensional locally compact quasi-simple loops for which the group generated by their left translations has a one-dimensional centre.

In the last section we use Betten's classification to determine in our framework the multiplicative loops $Q^{\ast }$ of the quasifields which coordinatize the $4$-dimensional non-desarguesian translation planes $\mathcal{A}$ admitting an at least seven-dimensional collineation group and to  study their properties. The results obtained there yield the following

\noindent
{\bf Theorem} {\it Let $\mathcal{A}$ be a  $4$-dimensional locally compact non-desarguesian translation plane which admits an at least $7$-dimensional collineation group $\Gamma$. If the quasifield $Q$ coordinatizing  $\mathcal{A}$  is constructed with respect to two lines such that their intersection points with the line at infinity are contained in the $1$-dimensional orbit of $\Gamma $ or contain the set of the fixed points of $\Gamma$, then the multiplicative loop $Q^{\ast }$ of  $Q$ is  decomposable if and only if one of the following cases occurs:\\
\noindent
(a) $\Gamma$ is $8$-dimensional, the translation complement $C$ is the group ${\mathrm{GL}}_2(\mathbb R)$ and acts reducibly on the translation group $\mathbb R^4$;
\newline
\noindent
(b) $\Gamma$ is $7$-dimensional, the translation complement $C$ fixes two distinct lines of $\mathcal{A}$ and leaves on one of them, one or two $1$-dimensional subspaces invariant;
\newline
\noindent
(c) $\Gamma$ is $7$-dimensional, the translation complement $C$ fixes two distinct lines $\{ S, W \}$ through the origin and acts transitively on the spaces $P_S$ and $P_W$ but does not act transitively on the product space $P_S \times P_W$, where
$P_S$ and $P_W$ are the sets of all $1$-dimensional subspaces of $S$, respectively of $W$. }

\bigskip
\centerline{\bf 2. Preliminaries}

\bigskip
\noindent
A binary system $(L, \cdot )$ is called a quasigroup if for any given $a, b \in L$ the
equations $a \cdot y=b$ and $x \cdot a=b$ have unique solutions which we denote by $y=a \backslash b$ and $x=b/a$. If a quasigroup $L$ has an element $1$ such that $x=1 \cdot x=x \cdot 1$ holds for all $x \in L$, then it is called a loop and $1$ is the identity element of $L$.  The left translations
$\lambda _a: L \to L, x \mapsto a \cdot x$ and the right translations $\rho _a: L \to L, x \mapsto x \cdot a$, $a \in L$, are bijections of $L$.
Two loops $(L_1, \circ)$ and $(L_2, \ast )$ are called isotopic if there exist three bijections
$\alpha , \beta , \gamma :L_1 \to L_2$ such that
$\alpha (x) \ast \beta (y)= \gamma (x \circ y)$ holds for all $x,y \in L_1$.
A binary system $(K, \cdot )$ is called a subloop of $(L, \cdot )$ if $K \subset L$, for any given $a, b \in K$ the
equations $a \cdot y=b$ and $x \cdot a=b$ have unique solutions in $K$ and $1 \in K$. The kernel of a homomorphism $\alpha :(L, \cdot ) \to (L', \ast )$ of a loop $L$ into a loop $L'$ is a normal subloop $N$ of $L$, i.e. a subloop of $L$ such that
\begin{equation} \label{normalloop} x \cdot N=N \cdot x, \ (x \cdot N) \cdot y= x \cdot (N \cdot y), \ (N \cdot x) \cdot y =N \cdot (x \cdot y)  \end{equation}
hold for all $x,y \in L$. A loop $L$ is called simple if $\{ 1 \}$ and $L$ are its only normal subloops. A loop $L$ is proper if it is not a group. 

\noindent
A loop $L$ is called topological, if it is a topological space and the binary operations $(a,b) \mapsto a \cdot b$,
$(a,b) \mapsto a \backslash b$, $(a,b) \mapsto b/a: L \times L \to L$ are continuous. Then the left and right translations of $L$ are homeomorphisms of $L$. If $L$ is a connected differentiable manifold such that the loop multiplication and the left division are continuously differentiable mappings, then we call $L$ an almost $\mathcal{C}^1$-differentiable loop. If also the right division of $L$ is continuously differentiable, then $L$ is a $\mathcal{C}^1$-differentiable loop.
A connected topological loop is quasi-simple if it contains no normal subloop of positive dimension.

\noindent
Every topological, respectively almost $\mathcal{C}^1$-differentiable, connected loop $L$ having a Lie group $G$ as the group topologically generated by the left translations of $L$ corresponds to a sharply transitive continuous, respectively  $\mathcal{C}^1$-differentiable section
$\sigma :G/H \to G$, where $G/H=\{ x H| x \in G\}$ consists of the left cosets of the stabilizer $H$
of $1 \in L$ such that $\sigma(H)=1_G$ and $\sigma (G/H)$ generates $G$.
The section $\sigma $ is sharply transitive if the image $\sigma (G/H)$ acts sharply transitively on the factor space $G/H$, i.e. for given left cosets $x H$, $y H$ there exists precisely one $z \in \sigma (G/H)$ which satisfies the equation $z x H=y H$.

A (left) quasifield  is an algebraic structure $(Q,+, \cdot )$ such that $(Q,+)$ is an abelian group with neutral element $0$,
$(Q \setminus \{ 0 \}, \cdot )$ is a loop with identity element $1$ and between these operations the (left) distributive law
$x \cdot (y+z) = x \cdot y + x \cdot z $ holds. A locally compact connected topological quasifield is a locally compact connected topological space $Q$ such that $(Q,+)$ is a topological group, $(Q \setminus \{ 0 \}, \cdot )$ is a topological loop, the multiplication $\cdot :Q \times Q \to Q$ is continuous and the mappings $\lambda_a: x \mapsto a \cdot x$ and $\rho _a: x \mapsto x \cdot a$ with $0 \neq a \in Q$ are homeomorphisms of $Q$.
If for any given $a,b,c \in Q$ the equation $x \cdot a  +  x \cdot b =c$ with $a+b \neq 0$ has  precisely one solution, then $Q$ is called planar. A translation plane is an affine plane with transitive group of translations; this is coordinatized by a planar quasifield (cf. \cite{pickert}, Kap. 8).

The kernel $K_r$ of a (left) quasifield $Q$ is a skewfield defined by
\begin{equation} \label{kernel}
K_r=\{k \in Q; \ (x+y) \cdot k = x \cdot k+ y \cdot k \ \hbox{and} \ (x \cdot y) \cdot k= x \cdot (y \cdot k) \ \hbox{for all} \ x,y \in Q\}. \nonumber \end{equation}

In this paper we consider left quasifields $Q$. Then $Q$ is a right vector space over $K_r$. Moreover, for all $a \in Q$ the map
$\lambda _a: Q \to Q, x \mapsto a \cdot x$ is $K_r$-linear.
According to \cite{hughes}, Theorem 7.3, p. 160, every quasifield that has finite dimension over its kernel is planar.

Let $F$ be a skewfield and let $V$ be a vector space over $F$. A collection ${\cal B}$ of subspaces of $V$ with $|{\cal B}| \ge 3$ is called a spread of $V$ if for any two different elements $U_1$, $U_2 \in {\cal B}$ we have $V= U_1 \oplus U_2$ and every vector of $V$ is contained in an element of ${\cal B}$.

If $S$ and $W$ are different subspaces of the spread ${\cal B}$, then $V$ can be coordinatized in such a way that 
$S= \{ 0 \} \times X$
 and $W=X \times \{ 0 \}$.  Any spread of $V=X \times X$ can be described by a collection ${\mathcal M}$ of linear mappings
$X \to X$ satisfying the following conditions:
\newline
\noindent
$(M_1)$ For any $\omega_1 \neq \omega_2 \in {\mathcal M}$ the mapping $\omega_1- \omega_2$ is bijective.
\newline
\noindent
$(M_2)$ For all $x \in X \setminus \{ 0 \}$ the mapping $\phi_x: {\mathcal M} \to X: \omega \mapsto \omega(x)$ is surjective.
\newline
\noindent
Namely, if ${\mathcal M}$ is a collection of linear mappings satisfying $(M_1)$ and $(M_2)$, then the sets 
$U_{\omega }=\{(x, \omega(x)), x \in X \}$ and $\{0 \} \times X$ yield a spread of $V=X \times X$. Conversely, every component $U \in {\cal B} \setminus \{ S \}$ of $V$ is the graph of a linear mapping $\omega_U: W \to S$ and the set of 
$\omega_U$ gives a collection  ${\mathcal M}$ of linear mappings of $X$ satisfying $(M_1)$ and $(M_2)$ 
(cf. \cite{knarr}, Proposition 1.11.). The mapping $\omega_W$ is the zero mapping.
For this reason any collection ${\mathcal M}$ of linear mappings of $X$ which satisfy $(M_1)$ and $M_2$ is called a spread set of $X$.
\newline
\noindent
Every translation plane can be obtained from a spread set of a suitable vector space $V=X \times X$ (cf. \cite{knarr}, Theorem 1.5, p. 7, and \cite{andre}). As every translation plane can be coordinatized by a quasifield and a quasifield contains $0$ and $1$, the associated spread set contains the zero endomorphism and the identity map.
This is not true for arbitrary spread sets ${\mathcal M}$, but if $\omega_0, \omega_1 \in {\mathcal M}$ are distinct, then
${\mathcal M}'=\{ (\omega -\omega_0)(\omega_1- \omega_0)^{-1}, \ \omega \in {\mathcal M}\}$ is  a normalized spread set of $X$, i.e. a spread set which contains the zero and the identity map. The translation planes obtained from 
${\mathcal M}$ and ${\mathcal M}'$ are isomorphic (cf. \cite{knarr}, Lemma 1.15, p. 13). 
Let ${\mathcal M}$ be a normalized spread set of $X$, $e \in X \setminus \{ 0 \}$ and let $\phi_e: {\mathcal M} \to X$ be defined by $\phi_e(\omega )= \omega(e)$. Then the multiplication $\circ: X \times X \to X$ defined by
$m \circ x= (\phi_e^{-1}(m))(x)$ yields a multiplicative loop of a left quasifield $Q$ coordinatizing the translation plane ${\cal A}$ belonging to the spread ${\mathcal M}$ of $X$.

\noindent
If we fix a basis of $Q$ over its kernel $K_r$ and identify $X$ with the vector space of pairs 
$\{(x,y)^t, \ x,y \in K_r \}$, then  the set
${\mathcal M}$ consists of matrices $C(\alpha, \beta, \gamma, \delta)=\left(\begin{array}{cc}
\alpha & \beta \\
\gamma & \delta \end{array} \right), \alpha, \beta, \gamma, \delta \in K_r$. If $e=(1,0)^t$, then we get $\phi_e( C(\alpha, \beta, \gamma,  \delta))=C(\alpha, \beta, \gamma, \delta)(e)=(\alpha, \gamma)^t$. Since ${\mathcal M}$ is a spread of $X$ the set of vectors
$(\alpha, \gamma)^t$ consists of all vectors of $X$. Hence if $(\alpha, \gamma)^t$ is an element of $X$, then there exists a unique matrix of ${\mathcal M}$ having $(\alpha, \gamma)^t$ as the first column.

We consider multiplicative loops of
locally compact connected  topological quasifields $Q$ of dimension $2$ coordinatizing $4$-dimensional non-desarguesian topological translation planes.
Then the kernel $K_r$ of $Q$
is isomorphic to the field of the real numbers, $(Q, +)$ is the vector group $\mathbb R^2$ and the multiplicative loop
$(Q \backslash \{ 0 \}, \cdot )$ is  homeomorphic to $\mathbb R \times S^1$,  where $S^1$ is the circle.

\bigskip
\noindent
\centerline{ \bf 4. Multiplicative loops of $2$-dimensional quasifields}

\bigskip
\noindent
Let $(Q,+, \ast )$ be a real topological (left) quasifield of dimension $2$. Let $e_1$ be the identity element of
the multiplicative loop $Q^{\ast }=(Q \setminus \{ 0 \}, \ast )$ of $Q$, which generates the kernel $K_r= \mathbb R$ of 
$Q$ as a vector space and let $B=\{e_1,e_2\}$ be a basis of the right vector space $Q$ over $K_r$. Once we fix $B$, we identify $Q$ with the vector space of pairs $(x,y)^t \in \mathbb R^2$
and $K_r$ with the subspace of pairs
$(x, 0)^t$. The element $(1, 0)^t$ is the identity element of $Q^{\ast }$.
According to \cite{loops}, Theorem 29.1, p. 345, the group $G$ topologically generated by the left translations of 
$Q^{\ast }$ is the connected component of the group ${\mathrm{GL}}_2(\mathbb R)$. As $\hbox{dim} \ Q^{\ast }=2$ and the stabilizer of the identity element of  $Q^{\ast }$ in $G$ does not contain any non-trivial normal subgroup  of $G$  we 
may replace the stabilizer of the identity 
by the subgroup
$H=\left\{ \left( \begin{array}{cc}
a & b \\
0 & a^{-1} \end{array} \right), a>0, b \in \mathbb R \right\}$.
The elements $g$ of $G$ have a unique decomposition as the product
\begin{equation}  g= \left( \begin{array}{cc}
u \cos t & u \sin t \\
-u \sin t & u \cos t \end{array} \right) \left( \begin{array}{cc}
k & l \\
0 & k^{-1} \end{array} \right) \nonumber \end{equation}  with suitable elements 
$u >0$, $k >0$,
$l \in \mathbb R$,
$t \in [0, 2 \pi )$. Hence the loop $Q^{\ast }$ corresponds to a continuous section
$\sigma : G/H \to G$;
\begin{equation}  \label{section} \left( \begin{array}{cc}
u \cos t & u \sin t \\
-u \sin t & u \cos t \end{array} \right) H \mapsto \left( \begin{array}{cc}
u \cos t & u \sin t \\
-u \sin t & u \cos t \end{array} \right) \left( \begin{array}{cc}
a(u, t) & b(u, t) \\
0 & a^{-1}(u, t) \end{array} \right) \end{equation}
\noindent
where the pair of continuous functions $a(u, t), b(u, t): \mathbb R_{>0} \times \mathbb [0,2 \pi) \to \mathbb R$, where 
$\mathbb R_{>0} $ is the set of positive numbers, satisfies the following conditions:
\begin{equation} a(u,t)>0, \ \ a(1,0)=1, \ \ b(1,0)=0. \nonumber \end{equation}

As $Q$ is a left quasifield, any $(x, y )^t \in Q^{\ast }$ induces a linear transformation $M_{(x,y)}\in \sigma(G/H)$. More precisely one has {\tiny \begin{equation}\label{quasifieldmultiplication} \left( \begin{array}{c}
x\\
y
\end{array} \right)* \left( \begin{array}{c}
u\\
v
\end{array} \right)=M_{(x,y)}\left( \begin{array}{c}
u\\
v
\end{array} \right)=
\left( \begin{array}{cc}
r \cos \varphi & r \sin \varphi \\
-r \sin \varphi & r \cos \varphi \end{array} \right)  \left( \begin{array}{cc}
a(r, \varphi ) & b(r, \varphi) \\
0 & a^{-1}(r, \varphi) \end{array} \right) \left(\begin{array}{c}
u\\
v
\end{array} \right), \end{equation}}
\noindent
where $x= r \cos(\varphi ) a(r, \varphi )$, \ $y=-r \sin(\varphi ) a(r, \varphi )$.
The kernel $K_r$ of $Q$ consists of $(0,0)^t$ and $(r \cos(k \pi) a(r,k \pi), 0)^t$, $r >0$, $k \in \{0,1\}$, such that the matrices corresponding to the elements $(r \cos(k \pi) a(r,k \pi), 0)^t$ have the form
\begin{equation} \label{quasifieldmultiplicationmas}
M(r \cos(k \pi) a(r,k \pi),0)=  \left( \begin{array}{cc}
r \cos(k \pi) a(r,k \pi) &  r \cos(k \pi) b(r,k \pi) \\
0 &  r \cos(k \pi) a^{-1}(r,k \pi) \end{array} \right).  \nonumber  \end{equation} 
The identity matrix $I$ corresponds to the identity $(1,0)^t$ of $Q^{\ast }$.
As to each real number $r \cos(k \pi) a(r,k \pi)$ belongs 
precisely one matrix 
\begin{equation} \label{matrixm} M(r \cos(k \pi) a(r,k \pi),0) \nonumber \end{equation}  the functions $f_1(r)=r a(r,0), r >0$ and $f_2(r)=-r a(r, \pi ), r >0$ are strictly monotone. If the functions $a(r,0)$, 
$a(r, \pi )$ are differentiable, then for every $r >0$ the derivatives $a(r,0) + r a'(r,0)$ and 
$-a(r, \pi )-r a'(r, \pi )$ are either always  positive or negative. This is equivalent to the fact that the derivatives 
$[\ln (a(r,0))]'$, $[\ln (a(r, \pi ))]'$ are always greater or smaller than $-r^{-1}$.

\begin{Rem} \label{diagonaluj}
The set ${\mathcal K}=\{ M(r \cos(k \pi) a(r,k \pi),0); r >0, k \in \{0,1\} \}$ of the left translations of $Q^{\ast }$ corresponding to the kernel $K_r$ of $Q$ is $$\left\{ \left( \begin{array}{cc}
r & 0 \\
0 & r \end{array} \right), r \in \mathbb R \setminus \{ 0 \} \right\}$$ if and only if one has $a(r,k \pi)=1$, 
$b(r,k \pi)=0$ for all $r >0$, $k \in \{0,1\}$.
\end{Rem}

The section $\sigma $ given by (\ref{section}) is sharply transitive precisely if for all pairs $(u_1, t_1)$, 
$(u_2, t_2)$ in
$\mathbb R_{>0} \times [0, 2 \pi )$ there exists precisely one $(u, t) \in \mathbb R_{>0} \times [0, 2 \pi )$ and $k>0$,
$l \in \mathbb R$ such that
\begin{equation} \label{equationmasodikuj} \left( \begin{array}{cc}
u \cos t & u \sin t \\
-u \sin t & u \cos t \end{array} \right)  \left( \begin{array}{cc}
a(u, t) & b(u, t) \\
0 & a^{-1}(u, t) \end{array} \right)  \left( \begin{array}{cc}
u_1 \cos t_1 & u_1 \sin t_1 \\
-u_1 \sin t_1 & u_1 \cos t_1 \end{array} \right)= \nonumber \end{equation}
\begin{equation} \left( \begin{array}{cc}
u_2 \cos t_2 & u_2 \sin t_2 \\
-u_2 \sin t_2 & u_2 \cos t_2 \end{array} \right) \left( \begin{array}{cc}
k & l \\
0 & k^{-1} \end{array} \right). \end{equation}
As the determinant of the matrices on both sides of (\ref{equationmasodikuj}) are equal we get that $u = u_1^{-1} u_2$.
Therefore the system (\ref{equationmasodikuj}) of equations is uniquely solvable if and only if for any fixed
$u >0$ the mapping
\begin{equation} \label{sectionu} \sigma _u: \left( \begin{array}{cc}
\cos t & \sin t \\
-\sin t & \cos t \end{array} \right) H \mapsto \left( \begin{array}{cc}
\cos t & \sin t \\
-\sin t & \cos t \end{array} \right)  \left( \begin{array}{cc}
a(u, t) & b(u, t) \\
0 & a^{-1}(u, t) \end{array} \right) \nonumber \end{equation}
determines a quasigroup $F_u$  homeomorphic to $S^1$.
One may take as the points of $F_u$ the vectors $(u a(u, t) a^{-1}(u, 0) \cos t, -u a(u, t) a^{-1}(u, 0) \sin t)^t$ and as the section the mapping
{\scriptsize \begin{equation} \label{sectionuuj} \sigma _u:  \left( \begin{array}{c}
 u a(u, t) a^{-1}(u, 0) \cos t \\
 -u a(u, t)a^{-1}(u, 0) \sin t \end{array} \right) \mapsto
\left( \begin{array}{cc}
\cos t & \sin t \\
-\sin t & \cos t \end{array} \right)  \left( \begin{array}{cc}
a(u, t) a^{-1}(u,0) & b(u, t) \\
0 & a^{-1}(u, t) a(u,0) \end{array} \right) = \nonumber \end{equation}
\begin{equation} \left( \begin{array}{cc}
a(u, t) a^{-1}(u, 0) \cos t & b(u,t) \cos t + a^{-1}(u, t) a(u, 0) \sin t \\
-a(u, t) a^{-1}(u, 0) \sin t & - b(u,t) \sin t +  a^{-1}(u, t) a(u, 0) \cos t \end{array} \right).
 \end{equation}}
In this way we see that the quasigroup $F_u$ has the right identity $(u, 0)^t$ since
{\scriptsize \begin{equation} \sigma _u \left( \begin{array}{c}
 u a(u, t) a^{-1}(u, 0) \cos t\\
 -u a(u, t)a^{-1}(u, 0) \sin t \end{array} \right) \cdot \left( \begin{array}{c} u\\ 0 \end{array} \right ) = \left( \begin{array}{c}
 u a(u, t) a^{-1}(u, 0) \cos t\\
 -u a(u, t)a^{-1}(u, 0) \sin t \end{array} \right). \nonumber \end{equation} }
The quasigroup $F_u$ is a loop, i.e. $(u, 0)^t$ is the left identity of $F_u$, if and only if
{\scriptsize $$\sigma _u \left( \begin{array}{c}
 u \\
0\end{array} \right)  = \left( \begin{array}{cc}
a(u, 0) a^{-1}(u, 0) \cos 0 & b(u,0) \cos 0  \\
0 &   a^{-1}(u, 0) a(u, 0) \cos 0 \end{array} \right) = \left( \begin{array}{cc}
1 & b(u,0)   \\
0 &  1\end{array} \right) =  \left( \begin{array}{cc}
1 & 0   \\
0 &  1\end{array} \right)$$}
\noindent
which means $b(u,0) = 0$ for all $u >0$. The almost ${\mathcal C}^1$-differentiable loop $Q^{\ast }$ belonging to the sharply transitive
${\mathcal C}^1$-differentiable section $\sigma $ given by (\ref{section}) is ${\mathcal C}^1$-differentiable precisely if the mapping
$(xH, yH) \mapsto z: G/H \times G/H \to \sigma(G/H)$ determined by $z xH=yH$ is
${\mathcal C}^1$-differentiable (cf. \cite{loops}, p. 32), i.e.
the solutions $u >0$, $t \in [0, 2 \pi )$ of the matrix equation  (\ref{equationmasodikuj}) are continuously differentiable functions of $u_1, u_2 \in \mathbb R_{>0}$, 
$t_1, t_2 \in [0,2 \pi)$. The function $u=u_1^{-1} u_2$ is continuously differentiable. If for each fixed
$u >0$ the section $\sigma_u$ given by (\ref{sectionuuj}) yields a $1$-dimensional 
$\mathcal{C}^1$-differentiable compact loop, then 
the function
$t(u_1,u_2,t_1,t_2)=t_{(u_1,u_2)}(t_1,t_2)$ is continuously differentiable (cf. \cite{loops}, Examples 20.3, p. 258). Indeed, the function $t_{(u_1,u_2)}(t_1,t_2)$ is determined implicitly by equations which depend continuously differentiably also on the parameters $u_1$ and $u_2$. 
Applying the above discussion we can prove the following:

\begin{Theo} \label{compactkernel}
Let $Q^{\ast }$ be the ${\mathcal C}^1$-differentiable multiplicative loop of a locally compact $2$-dimensional connected topological quasifield $Q$.
Then $Q^{\ast }$ is diffeomorphic to $S^1 \times \mathbb R$ and belongs to a ${\mathcal C}^1$-differentiable sharply transitive section $\sigma $ of the form
\begin{equation}  \left( \begin{array}{cc}
u \cos t & u \sin t \\
-u \sin t & u \cos t \end{array} \right) H \mapsto \left( \begin{array}{cc}
u \cos t & u \sin t \\
-u \sin t & u \cos t \end{array} \right) \cdot \left( \begin{array}{cc}
a(u, t) & b(u, t) \\
0 & a^{-1}(u, t) \end{array} \right), \nonumber \end{equation} with $b(u, 0)=0$ for all $u >0$
if and only if
for each fixed $u >0$ the function $a_u^{-1}(t) := a(u,0)a^{-1}(u, t)$ has the shape

\medskip
\noindent
\begin{equation}
a_u^{-1}(t)=  e^{t }(1- \int \limits _0^{t} R(s) e^{-s} \ ds )  \nonumber \end{equation}
\noindent
where $R(s)$ is a continuous function,
the Fourier series of which is contained in the set ${\cal F}$ of Definition 1 in \cite{figulastrambach} and converges
uniformly to $R$. Moreover, $b_u(t):=b(u, t)$ is a periodic ${\cal C}^1$-differentiable
function with $b_u(0)=b_u(2 \pi )=0$ such that

\noindent
\begin{equation}
b_u(t)> -a_u(t)  \int \limits _0^t \frac{(a^2_u(s)- a'^2_u(s))}{a^4_u(s)} \ d s \ \ \hbox{for \ all} \ \ 
t \in (0, 2 \pi ).  \nonumber \end{equation}
\end{Theo}
\begin{proof}
The section $\sigma_u$ given by (\ref{sectionuuj}) yields a $1$-dimensional $\mathcal{C}^1$-differentiable compact loop having the group ${\mathrm{SL}}_2(\mathbb R)$ as the group topologically generated by its left translations if and only if for each fixed
$u >0$ the continuously differentiable functions $a(u,0) a^{-1}(u, t):={\bar a}_u(t)$,
$-b(u, t):={\bar b}_u(t)$ satisfy the conditions
\begin{equation} \label{diffinequality}
{\bar a}'^2_u (t)+{\bar b}_u(t) {\bar a}'_u(t)+ {\bar b}'_u(t) {\bar a}_u(t) -{\bar a}^2_u(t) <0, 
\ {\bar b}'_u(0)< 1- {\bar a}'^2_u(0)  \end{equation}
\noindent
(cf. \cite{loops}, Section 18, (C), p. 238, \cite{figulastrambach}, pp. 132-139).
The solution of the differential inequalities (\ref{diffinequality}) is given by Theorem 6 in \cite{figulastrambach}, 
pp. 138-139. This proves the assertion.
\end{proof}

\begin{Prop}
Let $Q^{\ast }$ be the ${\mathcal C}^1$-differentiable multiplicative loop of a locally compact $2$-dimensional connected topological quasifield $Q$.
Assume that for each fixed $u >0$ the function $a_u(t):=a^{-1}(u,0) a(u, t)$ is the constant function $1$ and
that $b(u, 0) =0$ is satisfied for all
$u >0$.
Then $Q^{\ast }$ belongs to a ${\mathcal C}^1$-differentiable sharply transitive section $\sigma $ of the form 
(\ref{section}) if and only if for each fixed $u >0$ one has $b_u(t):= b(u,t) > -t$ 
for all 
$0 < t < 2 \pi $.
\end{Prop}
\begin{proof}
If for each fixed  $u > 0$ the function $a(u,0) a^{-1}(u, t)= a^{-1}_u(t)={\bar a}_u(t)$ is constant with value $1$, then
the section $\sigma_u$ given by (\ref{sectionuuj}) yields a ${\cal C}^1$-differentiable compact loop $L$  if and only if for each fixed 
$u >0$ the continuously differentiable function ${\bar b}_u(t):= -b_u(t)$ satisfies the differential inequality
${\bar b}'_u(t) < 1$ with the initial condition ${\bar b}'_u(0) < 1$
(cf. (\ref{diffinequality})). This is the case precisely if one has $b_u(t) >- t$ for all $0 < t < 2 \pi $.
\end{proof}

\begin{Prop}
Let $Q^{\ast }$ be the ${\mathcal C}^1$-differentiable multiplicative loop of a locally compact $2$-dimensional connected topological quasifield $Q$.
Assume that for each fixed $u > 0$ the function $b(u, t)$ is the constant function $0$.
Then $Q^{\ast }$ belongs to a ${\mathcal C}^1$-differentiable sharply transitive section $\sigma $ of the form 
(\ref{section}) precisely if for each fixed $u > 0$ one has 
$e^{-t}< a(u,t) a^{-1}(u,0) < e^{t}$ for all $0 < t < 2 \pi $.
\end{Prop}
\begin{proof}
If for each fixed  $u >0$ the function $b(u, t)=-{\bar b}_u(t)$ is constant with value $0$, then the section $\sigma_u$ given by (\ref{sectionu}) determines a ${\cal C}^1$-differentiable compact loop $L$  if and only if for each fixed
$u > 0$ 
the following inequalities are satisfied:
\begin{equation} \label{equequ4} ({\bar a}'_u(t) -{\bar a}_u(t))({\bar a}'_u(t) + {\bar a}_u(t)) <0,  \quad 0< 1- {\bar a}'^2_u(0), \nonumber  \end{equation}
where ${\bar a}_u(t)=a(u,0) a^{-1}(u,t)$.
This is the case precisely if either one has ${\bar a}'_u(t) -{\bar a}_u(t)<0$ and ${\bar a}'_u(t) + {\bar a}_u(t)>0$ or one has
${\bar a}'_u(t) -{\bar a}_u(t)>0$ and ${\bar a}'_u(t) + {\bar a}_u(t)<0$.
 Now we consider the first case. Then the function ${\bar a}_u(t)$ determines a loop if and only if for each fixed
 $u >0$ it is a subfunction of a differentiable function $h_u(t):=h(u,t)$ with $h_u(0)=1$, $h'^2_u(0)=1$,  
$h'_u(t)=h_u(t)$ and an upper function of a differentiable function $l_u(t):=l(u,t)$ with $l_u(0)=1$, $l'^2_u(0)=1$, 
$l'_u(t)=-l_u(t)$
 (cf. \cite{walter}, p. 66). Hence for each fixed $u > 0$ the function
${\bar a}_u(t)$ is a subfunction of the function $e^{t}$ and an upper function of the function $e^{-t}$ for all
$t \in (0, 2 \pi )$. Therefore, any continuously differentiable function ${\bar a}_u(t)$ such that for each fixed
$u > 0$ and for all
$t \in (0, 2 \pi )$ one has
$e^{-t}< {\bar a}_u(t)^{-1} < e^{t}$  determines a ${\mathcal C}^1$-differentiable compact loop $L$.

In the second case an analogous consideration as in the first case gives that for all fixed $u >0$ the function
$a(u,t) a^{-1}(u,0)$ must be a subfunction of the function $e^{-t}$ and an upper function of the function $e^{t}$ for all $t \in (0, 2 \pi )$. Hence in this case the function $a(u,t) a^{-1}(u,0)$ does not exist.
\end{proof}

\begin{Prop} \label{ujprop}
Let
{\tiny \begin{equation} \label{ujpropsection}
\left( \begin{array}{cc}
u \cos t & u \sin t \\
-u \sin t & u \cos t \end{array} \right) H \mapsto  \left( \begin{array}{cc}
u & 0 \\
0 & u \end{array} \right)
\left( \begin{array}{cc}
\cos t & \sin t \\
-\sin t & \cos t \end{array} \right) \left( \begin{array}{cc}
a(1, t) & b(u, t) \\
0 & a^{-1}(1, t) \end{array} \right), u >0, t \in \mathbb [0,2 \pi)  \end{equation} }
with $b(u, 0)=0$ for all $u >0$ be a section belonging to a multiplicative loop $Q^{\ast }$ of a locally compact 
$2$-dimensional connected topological quasifield $Q$. Then $Q^{\ast }$ contains for any
$u >0$ a $1$-dimensional compact subloop.
\end{Prop}
\begin{proof} The image of the section (\ref{ujpropsection}) acts sharply transitively on the point set 
$\mathbb R^2 \setminus \{ (0,0)^t \}$. Since the subgroup $\left\{ \left( \begin{array}{cc}
u & 0 \\
0 & u \end{array} \right), u > 0 \right\}$ leaves any line through $(0,0)^t$ fixed, the subset
\begin{equation} \label{compactsection} {\mathcal T}=\left\{\left( \begin{array}{cc}
\cos t & \sin t \\
-\sin t & \cos t \end{array} \right) \left( \begin{array}{cc}
a(1, t) & b(u, t) \\
0 & a^{-1}(1, t) \end{array} \right), t \in \mathbb [0,2 \pi) \right\} \end{equation}
acts sharply transitively on the oriented lines through $(0,0)^t$ for any $u > 0$.
Therefore ${\mathcal T}$ corresponds to a $1$-dimensional compact loop since $b(u,0)=0$ for all $u >0$. \end{proof}

\medskip
\noindent
As ${\mathcal T}$ given by (\ref{compactsection}) is the image of a section corresponding to a $1$-dimensional compact subloop of $Q^{\ast }$, every element of ${\mathcal T}$ is elliptic.

\begin{Prop}
Every element of the set ${\mathcal T}$ given by (\ref{compactsection}) is elliptic if and only if the following holds:  
$a(1, k \pi)=1$ for $k \in \{0, 1\}$:
\newline
\noindent
1) if for all $t \in [0, 2 \pi)$ and $u > 0$ one has $b(u,t)=0$, then the function $a(1,t)$ satisfies the inequalities:
{ \begin{equation} \label{ineqcompact} \frac{1- |\sin(t)|}{|\cos(t)|} \le a(1,t) \le \frac{1+ |\sin(t)|}{|\cos(t)|}, 
\end{equation} }
2) if the function  $b(u,t)$ is different from the constant function $0$, then for $\sin(t)>0$ one has
{\small\begin{equation} \label{inequj} \frac{(a(1,t)+ a(1,t)^{-1}) \cos (t)-2}{\sin(t) } < b(u,t) <  \frac{(a(1,t)+ a(1,t)^{-1}) \cos (t)+2}{\sin(t) }, \  \end{equation} }
for $\sin(t)<0$ we have
{\small\begin{equation} \label{inequjuj} \frac{(a(1,t)+ a(1,t)^{-1}) \cos (t)+2}{\sin(t) } < b(u,t) <  \frac{(a(1,t)+ 
a(1,t)^{-1}) \cos (t)-2}{\sin(t) }. \end{equation} }
\end{Prop}
\begin{proof}
Any element of (\ref{compactsection}) is elliptic if and only if the inequality
\begin{equation} \label{ineqelso} |\cos (t) (a(1,t)+a(1,t)^{-1}) -\sin (t) b(u,t)| \le 2 \end{equation}
holds, where the equality sign occurs only for $t=k \pi $, $k \in \{0, 1\}$. Hence $a(1,k \pi)=1$. 
If $b(u,t)=0$, then  inequality (\ref{ineqelso}) reduces to $a^2(1,t) |\cos (t)|- 2 a(1,t) +|\cos (t)| \le 0$ which is equivalent to inequalities (\ref{ineqcompact}).
If $b(u,t) \neq 0$, then inequality (\ref{ineqelso}) is equivalent for all $t \neq k \pi$, $k \in \{0, 1\}$, to
{\tiny \begin{equation} \label{ineqmasodik}
(a(1,t)+ a(1,t)^{-1})^2 \cos^2(t) -2 (a(1,t)+ a(1,t)^{-1})  \sin (t) \cos (t) b(u,t) + \sin^2(t) b^2(u,t) < 4. 
\end{equation} }
Solving the quadratic equation
{\tiny \begin{equation} \label{ineqquadratic}
(a(1,t)+ a(1,t)^{-1})^2 \cos^2(t) -2 (a(1,t)+ a(1,t)^{-1}) \sin (t) \cos (t) x + \sin^2(t) x^2 = 4 \end{equation} }
we get
{\tiny \begin{equation} \label{solution}
x= \frac{2 (a(1,t)+ a(1,t)^{-1}) \cos(t) \sin (t) \pm 4 \sin (t)}{2 \sin^2(t)}= 
\frac{(a(1,t)+ a(1,t)^{-1}) \cos(t) \pm 2}{\sin (t)}. \nonumber \end{equation} }
Comparing (\ref{ineqmasodik}) and (\ref{ineqquadratic}) one obtains
{\tiny \begin{equation} \label{ineqharmadik} \left( b(u,t)- \frac{(a(1,t)+ a(1,t)^{-1}) \cos(t) - 2}{\sin (t)} \right)
\left( b(u,t)- \frac{(a(1,t)+ a(1,t)^{-1}) \cos(t) + 2}{\sin (t)} \right) < 0 \nonumber \end{equation} }
which yields inequalities (\ref{inequj}) and (\ref{inequjuj}).
\end{proof}

\begin{Prop} \label{complex} Let $Q^{\ast }$ be the multiplicative loop of a  locally compact quasifield $Q$ of dimension $2$ containing a $1$-dimensional compact normal subloop.
The quasifield $Q$ is the field $\mathbb C$ of complex numbers if and only if ${\mathcal T}$  is a normal subset in the set of all left translations of $Q^{\ast }$.
\end{Prop}
\begin{proof}
If $Q$ is the field of complex numbers, then $Q^{\ast }$ is the group
${\mathrm{SO}}_2(\mathbb R) \times \mathbb R$ and the assertion is true. If the set ${\mathcal T}$ is a normal subset in the set of the left translations of a proper loop $Q^{\ast }$, then it is normal in the connected component 
${\mathrm{GL}}_2^+(\mathbb R)$ of the group 
${\mathrm{GL}}_2(\mathbb R)$ because ${\mathrm{GL}}_2^+(\mathbb R)$ is the group topologically generated by the left translations of 
$Q^{\ast }$ (cf. \cite{loops}, Section 29, p. 345).
If ${\mathcal T}$ is normal in ${\mathrm{GL}}_2^+(\mathbb R)$, then  for $D=\left( \begin{array}{cc}
\cos \varphi & \sin \varphi \\
-\sin \varphi & \cos \varphi \end{array} \right)$ one has that
{\scriptsize \begin{equation}
D^{-1} \left( \begin{array}{cc}
\cos t & \sin t \\
-\sin t & \cos t \end{array} \right) D D^{-1} \left( \begin{array}{cc}
a(1,t) & b(1,t) \\
0 & a^{-1}(1,t) \end{array} \right) D= \nonumber \end{equation}
\begin{equation} \left( \begin{array}{cc}
\cos t & \sin t \\
-\sin t & \cos t \end{array} \right) D^{-1} \left( \begin{array}{cc}
a(1,t) & b(1,t) \\
0 & a^{-1}(1,t) \end{array} \right) D=
\left( \begin{array}{cc}
\cos t & \sin t \\
-\sin t & \cos t \end{array} \right)
\left( \begin{array}{cc}
a(1,t) & b(1,t) \\
0 & a^{-1}(1,t) \end{array} \right) \nonumber \end{equation} }
is satisfied for all $\varphi \in [0, 2 \pi)$ if and only if $a(1,t)=1$ and $b(1,t)=0$ or equivalently 
${\mathcal T}={\mathrm{SO}}_2(\mathbb R)$.
But the compact group ${\mathrm{SO}}_2(\mathbb R)$ is not normal in ${\mathrm{GL}}_2^+(\mathbb R)$. Hence 
$Q^{\ast }$ is not proper and the assertion follows. \end{proof}

\begin{Lemma} \label{nonquasi} If the multiplicative proper loop $Q^{\ast }$ of a $2$-dimensional locally compact connected topological quasifield $Q$ is not quasi-simple, then the set 
${\mathcal K}=\{M(r \cos (k \pi) a(r, k \pi), 0); r > 0, k \in \{0,1 \} \}$ of the  left translations
of $Q^{\ast }$ belonging to the kernel $K_r$ of $Q$ has the form
$\left\{ \left( \begin{array}{cc}
r  & 0 \\
0 &  r   \end{array} \right), 0 \neq r\in \mathbb R \right\}$, which is a normal subgroup of the set 
$\Lambda _{Q^{\ast }}$ of all left translations of $Q^{\ast }$.
\end{Lemma}
\begin{proof} 
By Lemma 1.7, p. 19, in \cite{loops}, the left translations of a normal subloop of $Q^{\ast }$ generate a normal subgroup $N$ of  ${\mathrm{GL}}_2^+(\mathbb R)$ which is the group topologically generated by all left translations of 
$Q^{\ast }$. Hence the set 
$\Lambda _{Q^{\ast }}$ of the left translations of 
$Q^{\ast }$ must contain the group $C =\left\{ \left( \begin{array}{cc}
u & 0  \\
0 & u \end{array} \right), \ 0 < u \in \mathbb R  \right\}$ as a normal subgroup. It follows that the set ${\mathcal K}$ of the left translations corresponding to the elements of the kernel
$K_r$ of $Q$ contained in  $\Lambda _{Q^{\ast }}$ has the form given in the assertion and 
${\cal K} \cap \Lambda _{Q^{\ast }}$ is a normal subgroup of $\Lambda _{Q^{\ast }}$.
\end{proof}

\noindent
Assume that the set ${\cal K}$ of the left translations of the loop $Q^{\ast }$ having $(1,0)^t$ as identity corresponding to the elements of the kernel $K_r$ of $Q$ has the form given in Lemma \ref{nonquasi}. According to 
(\ref{quasifieldmultiplication}) the element
$$ \left( \begin{array}{cc}
ra(r, \varphi ) \cos \varphi & r b(r, \varphi) \cos \varphi + ra^{-1}(r, \varphi) \sin \varphi \\
-ra(r, \varphi ) \sin \varphi & -r b(r, \varphi)\sin \varphi + ra^{-1}(r, \varphi)\cos \varphi \end{array} \right) $$

\noindent
corresponds to the left translation of  $(r a(r, \varphi ) \cos \varphi ,-r a(r, \varphi ) \sin \varphi )^t$, $r >0$, 
$\varphi \in [0, 2 \pi)$. 
Let $N^{\ast}$ be the subgroup of $Q^{\ast}$ corresponding to the normal subgroup
${\mathcal K}$ of $\Lambda _{Q^{\ast }}$. We show that
$N^{\ast }:=\{ ({\hat s}, 0)^t, {\hat s} \in \mathbb R \setminus \{ 0 \} \}$ is normal in $Q^{\ast}$.
For all elements $x:= (\cos \varphi , -\sin \varphi )^t$, $y:=(u,v)^t$ of $Q^{\ast}$ the condition
$(N^{\ast } \ast x) \ast y=N^{\ast } \ast (x \ast y)$ of (\ref{normalloop}) is satisfied if and only if
we have
{\tiny \begin{equation} \left[ \left( \begin{array}{c}
{\hat s} \\
0 \end{array} \right) \ast \left( \begin{array}{c}
 cos \varphi \\
-sin \varphi  \end{array} \right)\right]
\ast \left( \begin{array}{c}
u \\
v\end{array} \right) = \left( \begin{array}{c}
{\hat s}' \\
0 \end{array} \right) \ast \left[ \left( \begin{array}{c}
 cos \varphi \\
-sin \varphi  \end{array} \right)
\ast \left( \begin{array}{c}
u \\
v\end{array} \right)\right] \nonumber \end{equation}} for all $\varphi \in [0, 2 \pi)$, 
$(u,v) \in \mathbb R^2 \setminus \{(0,0) \}$ with suitable ${\hat s}, {\hat s}' \in \mathbb R \setminus \{ 0 \}$ such that ${\hat s}=s \cos(k \pi)$, 
${\hat s}'=s' \cos(k \pi)$, where $s, s'>0$, $k \in \{ 0, 1\}$.
This is the case precisely if  one has
{\tiny \begin{equation} \left( \begin{array}{c}
u s a(s, \varphi + k \pi) \cos (\varphi + k \pi)  + vs b(s, \varphi+ k \pi) \cos (\varphi +k \pi)+ 
v s a^{-1}(s, \varphi +k \pi) \sin (\varphi + k \pi) \\
-u s a(s, \varphi + k \pi ) \sin (\varphi + k \pi) -vs b(s, \varphi + k \pi)\sin (\varphi + k \pi) + 
v s a^{-1}(s, \varphi + k \pi ) \cos (\varphi + k \pi) \end{array} \right) = \nonumber \end{equation} 
\begin{equation} \left( \begin{array}{c}
s' \cos (k \pi )(u a(1, \varphi ) \cos \varphi + v b(1, \varphi) \cos \varphi + v a^{-1}(1, \varphi) \sin \varphi ) \\
s' \cos (k \pi )(-u a(1, \varphi ) \sin \varphi -v b(1, \varphi) \sin \varphi + v a^{-1}(1, \varphi) \cos \varphi ) 
\end{array} \right) \nonumber \end{equation} }
or equivalently for all $\varphi \in [0, 2 \pi )$, $(u,v) \in \mathbb R^2 \setminus \{(0,0) \}$ we have
{\tiny \begin{equation}
[(u a(s, \varphi + k \pi)+ v b(s, \varphi + k \pi)) \cos (\varphi+ k \pi) + 
v a^{-1}(s, \varphi+ k \pi) \sin (\varphi+ k \pi)] \cdot [(-u a(1, \varphi ) 
-v b(1, \varphi))\sin \varphi + v a^{-1}(1, \varphi)\cos \varphi] = \nonumber \end{equation}
\begin{equation}
[(-u a(s, \varphi + k \pi)-v b(s, \varphi+ k \pi)) \sin (\varphi+ k \pi) + 
v a^{-1}(s, \varphi+ k \pi) \cos (\varphi+ k \pi)] \cdot [(u a(1, \varphi ) + v b(1, \varphi)) \cos \varphi + 
v a^{-1}(1, \varphi) \sin \varphi ]. \nonumber \end{equation}}
The last equation holds if and only if one has 
{\tiny \begin{equation}
(a(s, \varphi+ k \pi)a^{-1}(1, \varphi) - a^{-1}(s, \varphi+ k \pi)a(1, \varphi) )uv + 
(b(s, \varphi+ k \pi)a^{-1}(1, \varphi) - a^{-1}(s, \varphi+ k \pi)b(1, \varphi ))v^{2} = 0. \nonumber \end{equation} }
Therefore we obtain $a(s, \varphi+ k \pi) a^{-1}(1, \varphi) - a^{-1}(s, \varphi+ k \pi) a(1, \varphi )= 0$ and 
$b(s, \varphi+ k \pi) a^{-1}(1, \varphi) - a^{-1}(s, \varphi+ k \pi) b(1, \varphi) =0$. 
As $a(s, \varphi )$ is positive  we have $a(s, \varphi+ k \pi) =a(1, \varphi)$,  
$b(s, \varphi + k \pi) =b(1, \varphi )$ for all
$s >0$, $\varphi \in [0, 2 \pi)$, $k \in \{0, 1\}$. 
By (\ref{normalloop}) the group $N^{\ast }$ is a normal subgroup of $Q^{\ast }$ if and only if for all 
$\varphi \in [0, 2 \pi)$, 
$(u,v) \in \mathbb R^2 \setminus \{(0,0)\}$ one has
{\tiny \begin{equation} \left[\left( \begin{array}{c}
 \cos \varphi \\
- \sin \varphi \end{array} \right)
\ast \left( \begin{array}{c}
{\hat s} \\
0 \end{array} \right)\right] \ast \left( \begin{array}{c}
 u \\
v \end{array} \right) =  \left( \begin{array}{c}
 \cos \varphi \\
- \sin \varphi \end{array} \right) \ast \left[ \left( \begin{array}{c}
{\hat s}' \\
0 \end{array} \right) \ast \left( \begin{array}{c}
u \\
v \end{array} \right) \right ] \ \ \hbox{or} \nonumber \end{equation} }
{\tiny \begin{equation} \left( \begin{array}{cc}
s a(1, \varphi)a(1, \varphi) \cos (\varphi+k \pi) & s a(1, \varphi)b(1, \varphi) \cos (\varphi+k \pi) + 
s sin (\varphi+k \pi)\\
-sa(1, \varphi)a(1, \varphi) \sin (\varphi+ k \pi) & -s a(1, \varphi)b(1, \varphi) \sin (\varphi+ k \pi) + 
s cos (\varphi+ k \pi) \end{array} \right)
\left( \begin{array}{c}
u  \\
v \end{array} \right) = \nonumber \end{equation} 
\begin{equation} \left( \begin{array}{cc}
a(1, \varphi)\cos \varphi & b(1, \varphi)\cos \varphi + a^{-1}(1, \varphi) sin \varphi\\
-a(1, \varphi) \sin \varphi & - b(1, \varphi)\sin \varphi + a^{-1}(1, \varphi) cos \varphi \end{array} \right)
\left( \begin{array}{c}
s' \cos(k \pi) u \\
s' \cos(k \pi) v \end{array} \right) \nonumber \end{equation} }
for suitable ${\hat s}, {\hat s}' \in \mathbb R \setminus \{0 \}$ such that ${\hat s}=s \cos(k \pi)$, 
${\hat s}'=s' \cos (k \pi)$, where $s, s' >0$, $k \in \{0, 1\}$. This is equivalent to
{\tiny \begin{equation}\left( \begin{array}{c}
sua(1, \varphi )^{2} \cos (\varphi+ k \pi) + sv[a(1, \varphi ) b(1, \varphi ) \cos (\varphi+k \pi) +  
\sin (\varphi+k \pi) ]\\
-su a(1, \varphi )^{2} \sin (\varphi+ k \pi) + sv[-a(1, \varphi ) b(1, \varphi ) \sin (\varphi+ k \pi) +  
\cos (\varphi+k \pi) ]\end{array} \right)
  = \nonumber \end{equation} 
\begin{equation}	\left( \begin{array}{c}
us'a(1, \varphi)  \cos (\varphi+k \pi)  + s'v[b(1, \varphi) \cos (\varphi+k \pi) + a^{-1}(1, \varphi) 
\sin (\varphi+k \pi)] \\
-us'a(1, \varphi) \sin (\varphi+k \pi) + s'v[-b(1, \varphi) \sin (\varphi+k \pi) + a^{-1}(1, \varphi) 
\cos (\varphi+k \pi)] \end{array} \right). \nonumber \end{equation} }
A direct computation yields that the above equality is true for all $(u,v) \in \mathbb R^2 \setminus \{(0,0)\}$, 
$\varphi \in [0, 2 \pi )$, $k \in \{0, 1 \}$. 

\medskip
\noindent 
Using Proposition \ref{complex}, Lemma \ref{nonquasi} and the discussion above we have the following

\begin{Theo} \label{normal}
The multiplicative  loop $Q^{\ast }$ of a locally compact $2$-dimensional quasifield $Q$ with $(1,0)^t$ as identity 
of $Q^{\ast }$ is proper and not quasi-simple if and only if  for all
$r > 0$, $\varphi \in [0, 2 \pi)$, $k \in \{0, 1\}$ one has $a(r,k \pi) = 1$, $b(r,k \pi) = 0$, 
$a(r, \varphi+k \pi) = a(1, \varphi)$ and 
$b(r, \varphi+k \pi) = b(1, \varphi)$.
Then $Q^{\ast }$ is a split extension of a $1$-dimensional normal subgroup $N^{\ast }$ isomorphic to 
$\mathbb R \times Z_2$, where $Z_2$ is the group of order $2$ by a subloop homeomorphic to the 
$1$-sphere.  
\end{Theo}
\begin{proof} We have only to prove the last claim.
According to Lemma \ref{nonquasi} and the above discussion the only possibility for a normal subloop of positive dimension is the group $N^{\ast }$. The intersection of a compact subloop of $Q^{\ast }$ with $N^{\ast }$ has cardinality $2$ (cf. Proposition \ref{ujprop}  and Lemma \ref{nonquasi}). Hence the claim is proved. \end{proof}

\medskip
\noindent 
From Proposition \ref{complex} and Theorem \ref{normal} it follows: 

\begin{Coro} 
The multiplicative  loop $Q^{\ast }$ of a locally compact $2$-dimensional quasifield $Q$ with $(1,0)^t$ as identity 
of $Q^{\ast }$ is the direct product of the group $\mathbb R$ and a subloop homeomorphic to the 
$1$-sphere if and only if $Q$ is the field of complex numbers. 
\end{Coro}

\bigskip 
\noindent
\centerline{ \bf 5. Decomposable multiplicative loops of $2$-dimensional quasifields}

\bigskip
\noindent
\begin{Dfn}
We call the multiplicative loop $Q^{\ast }$ of a locally compact connected topological $2$-dimensional quasifield $Q$
{\em decomposable}, if the set of all left translations of $Q^{\ast }$ is a product ${\mathcal T} {\mathcal K}$
with  $|{\mathcal T} \cap {\mathcal K}| \leq 2$, where
${\mathcal T}$ is the set of all left translations of a $1$-dimensional compact loop of the form \emph{(\ref{compactsection})} and ${\mathcal K}$ is the set of all left translations of $Q^{\ast }$ belonging to the kernel $K_r$ of 
$Q$.
\end{Dfn}

\noindent
If the loop $Q^{\ast }$ is decomposable, then it contains compact subloops for any
$u >0$ corresponding to the section (\ref{ujpropsection}). From now on we choose $u=1$. Then the compact subloop $F_u$ of 
$Q^{\ast }$ has identity $(1,0)^t$ and one has
{\tiny \begin{equation} \label{equproduct}
\left( \begin{array}{rr}
\cos t a(1,t) & \cos t b(1,t)+ \sin t a^{-1}(1,t) \\
- \sin t  a(1,t) & - \sin t b(1,t)+ \cos t a^{-1}(1,t) \end{array} \right)
\left[ \left( \begin{array}{cc}
r \cos (k \pi ) a(r, k \pi ) & r \cos (k \pi ) b(r,k \pi ) \\
0 & r \cos (k \pi ) a^{-1}(r, k \pi ) \end{array} \right) \left(\!\!\begin{array}{c}
1\\
0
\end{array}\!\!\right) \right] =  \nonumber \end{equation}
\begin{equation}
=\left( \begin{array}{rr}
r \cos (t+k \pi ) a(r,t+k \pi ) & r \cos (t+k \pi ) b(r,t+k \pi )+ r \sin (t+k \pi ) a^{-1}(r,t+k \pi ) \\
- r \sin (t+k \pi ) a(r,t+k \pi ) & - r \sin (t+k \pi ) b(r,t+k \pi )+ r \cos (t+k \pi ) a^{-1}(r,t+k \pi ) \end{array} \right)
\left(\!\!\begin{array}{c}
1\\
0
\end{array}\!\!\right).  \end{equation} }
Equation (\ref{equproduct}) yields that $a(r,t+k \pi )= a(1,t) a(r, k \pi)$.

\noindent
Now we give sufficient and necessary conditions for the loop $Q^{\ast }$ to be decomposable.

\begin{Prop} \label{productloop}
The multiplicative loop $Q^{\ast }$ of a locally compact connected topological $2$-dimensional quasifield $Q$ with 
$(1,0)^t$ as identity of $Q^{\ast }$ is decomposable if and only if for all $r > 0$,
$t \in [0, 2 \pi)$, $k \in \{0, 1\}$ one has 
\begin{equation} a(r,t+k \pi)= a(1,t) a(r, k \pi), \  
b(r,t+ k \pi)=a(1,t) b(r,k \pi)+ a^{-1}(r, k \pi) b(1,t). \nonumber \end{equation}
\end{Prop}
\begin{proof} The point $(x,y)^t$ is the image of the point $(1,0)^t$ under the linear mapping $M_{(x,y)}$ and the set
$\{ M _{(x,y)}; (x,y)^t \in Q^{\ast } \}$ acts sharply transitively on
$Q^{\ast }$.
The matrix equation 
{\tiny \begin{equation} \label{equproductmas}
 \left(\!\!\begin{array}{rr}
\cos t a(1,t) & \cos t b(1,t)+ \sin t a^{-1}(1,t) \\
- \sin t  a(1,t) & - \sin t b(1,t)+ \cos t a^{-1}(1,t) \end{array}\!\!\right)
\left[\!\!\left(\!\!\begin{array}{cc}
r \cos (k \pi) a(r, k \pi) & r \cos (k \pi) b(r,k \pi) \\
0 & r \cos (k \pi) a^{-1}(r, k \pi) \end{array}\!\!\right)
\left(\!\!\begin{array}{r}
u \cos \varphi a(u, \varphi )\\
- u \sin \varphi a(u, \varphi ) \end{array}\!\!\right)\!\!\right]= \nonumber \end{equation}
\begin{equation}
=\left(\!\!\begin{array}{rr}
r \cos (t+ k \pi ) a(r,t+ k \pi ) & r \cos (t+ k \pi ) b(r,t+ k \pi )+ r \sin (t+ k \pi ) a^{-1}(r,t+ k \pi ) \\
- r \sin (t+ k \pi )  a(r,t+ k \pi) & - r \sin (t+ k \pi ) b(r,t+ k \pi )+ r \cos (t+ k \pi ) a^{-1}(r,t+k \pi ) 
\end{array}\!\!\right) \left(\!\!\begin{array}{r}
u \cos \varphi a(u, \varphi )\\
- u \sin \varphi a(u, \varphi ) \end{array}\!\!\right) \end{equation} }
holds precisely if the identities of the assertion are satisfied.
\end{proof}

\noindent
If $Q^{\ast }$ is decomposable such that its compact subloop has identity $(1,0)^t$,  then 
$|{\mathcal T} \cap {\mathcal K}| = 2$ because one has 
${\mathcal T} \cap {\mathcal K}=\left\{ I, \left(\begin{array}{cc}
-1 & - b(1,\pi) \\
0 & -1 \end{array} \right) \right\}$. 
In this case the set of all left translations of $Q^{\ast }$ is a product ${\mathcal T} {\cal W}$
with  ${\mathcal T} \cap {\cal W} = I$, where  $\cal{W}$ is the set of all left translations corresponding to the connected component of the kernel $K_r$ of $Q$. 
\begin{Theo} \label{quasisimple}
If the multiplicative loop $Q^{\ast }$ of a locally compact connected topological $2$-dimensional quasifield $Q$ with 
$(1,0)^t$ as identity of $Q^{\ast }$ is not quasi-simple, then
$Q^{\ast }$ is decomposable.
\end{Theo}
\begin{proof}
By Theorem \ref{normal} the loop $Q^{\ast }$ is not quasi-simple if and only if for all
$r > 0$, 
$t \in [0, 2 \pi)$, $k \in \{0, 1 \}$ one has  $a(r, k \pi) = 1$, $b(r, k \pi) = 0$, $a(r, t+ k \pi) = a(1, t)$ and 
$b(r, t+ k \pi) = b(1, t)$.
Therefore the identities given in the assertion of Proposition \ref{productloop} are satisfied.   
\end{proof}

According to Theorems \ref{normal}, \ref{quasisimple} and the above discussion  
the set $\Lambda _{Q^{\ast }}$ of the left translations of $Q^{\ast }$ with a normal subloop of positive dimension and with $(1,0)^t$ as identity can be written into the form
\begin{equation} \label{diagonal}
\left\{ \left( \begin{array}{cc}
\cos t & \sin t \\
-\sin t & \cos t \end{array} \right) \left( \begin{array}{cc}
u a(1,t) & u b(1,t)  \\
0 & u a^{-1}(1,t) \end{array} \right), u > 0, t \in [0, 2 \pi ) \right\} \end{equation}
with $a(1, k \pi )=1$, $b(1, k \pi)=0$, $k \in \{0, 1 \}$.

\begin{Prop}
The ${\cal C}^1$-differentiable multiplicative loop $Q^{\ast }$ of a locally compact connected topological $2$-dimensional quasifield $Q$ with 
$(1,0)^t$ as identity of $Q^{\ast }$ is decomposable precisely if for the inverse function
$\bar{a}(1, t)=a^{-1}(1, t)$ and for $\bar{b}(1, t)=- b(1, t)$ the differential inequalities
\begin{equation} \label{diffinequalityuj}
\bar{a}'^2 (1,t)+\bar{b}(1,t) \bar{a}'(1,t)+\bar{b}'(1,t) \bar{a}(1,t) -\bar{a}^2(1,t) <0,\ 
\bar{b}'(1, 0)< 1- \bar{a}'^2(1, 0)  \end{equation}
are satisfied.
\end{Prop}
\begin{proof}
If $Q^{\ast }$ is a ${\cal C}^1$-differentiable multiplicative loop of a quasifield $Q$ with 
$(1,0)^t$ as identity, then the continuously differentiable functions 
${\bar a}_u(t):=a(u,0) a^{-1}(u, t)$, ${\bar b}_u(t):=-b(u,t)$ satisfy the conditions in (\ref{diffinequality}).
The set of all left translations of $Q^{\ast }$ is a product ${\mathcal T} {\mathcal K}$ if and only if one has 
$a(u, t+k \pi)= a(u,k \pi) a(1, t)$ and
$b(u, t+k \pi)=a(1, t) b(u,k \pi)+ a^{-1}(u,k \pi) b(1, t)$ for all $u>0$, $t \in [0,2 \pi)$, $k\in \{0,1\}$ 
(cf. Proposition \ref{productloop}). Putting this into 
(\ref{diffinequality}) we get for $u=1$ 
\begin{equation} \label{diffinequalitymas}
a'^2 (1,t)+ b(1,t) a'(1,t) a^2(1,t)- b'(1,t) a^3(1,t) - a^2(1,t) <0 \ \hbox{and} \  \nonumber \end{equation}
\begin{equation} b'(1,0)> a'^2(1,0)-1  \end{equation}
since $a(1,k \pi)=1$, $k\in \{0,1\}$. Inequalities (\ref{diffinequalitymas}) are equivalent to the inequalities 
(\ref{diffinequalityuj}) with 
$\bar{a}(1,t)=a^{-1}(1,t)$ and
$\bar{b}(1,t)=- b(1,t)$.
\end{proof}

\begin{Coro}
Let $T$ be any  $1$-dimensional ${\cal C}^1$-differentiable connected compact loop such that the set ${\mathcal T}$ of its left translations has the form \emph{(\ref{compactsection})} and let ${\mathcal K}$ be any set of matrices of the form
\begin{equation}
{\mathcal K}=  \left\{\left( \begin{array}{cc}
u \cos (k \pi) a(u, k \pi) & u \cos (k \pi ) b(u,k \pi) \\
0 &  u \cos (k \pi) a^{-1}(u, k \pi) \end{array} \right), u >0, k \in \{0, 1 \} \right\}, \nonumber \end{equation}
where  $a(u,k \pi)>0$ and $b(u,k \pi)$ are continuously differentiable functions such that  
$u a(u,0)$, $-u a(u,\pi)$ are strictly monotone. Then
the product ${\mathcal T} {\mathcal K}$ is the set of all left translations of a ${\cal C}^1$-differentiable decomposable multiplicative loop
$Q^{\ast }$ of a $2$-dimensional locally compact connected quasifield $Q$ with $(1,0)^t$ as identity of $Q^{\ast }$. 
\end{Coro}
\begin{proof}
Since one has 
{\tiny \begin{equation} \label{productuj1}
\left( \begin{array}{cc}
\cos t & \sin t \\
-\sin t & \cos t \end{array} \right) \left( \begin{array}{cc}
a(1,t) & b(1,t) \\
0 & a(1,t)^{-1} \end{array} \right) \left( \begin{array}{cc}
u \cos (k \pi) a(u, k \pi) & u \cos (k \pi) b(u, k \pi) \\
0 & u \cos (k \pi) a^{-1}(u, k \pi) \end{array} \right)=  \nonumber \end{equation}
\begin{equation} \left( \begin{array}{cc}
\cos (t+ k \pi) & \sin (t+ k \pi) \\
-\sin (t+ k \pi) & \cos (t+ k \pi) \end{array} \right) \left( \begin{array}{cc}
u a(u, k \pi) a(1,t) & u b(u,k \pi) a(1,t)+ u b(1,t) a^{-1}(u, k \pi) \\
0 & u a^{-1}(u, k \pi) a(1,t)^{-1} \end{array} \right) = \nonumber \end{equation}
\begin{equation} \label{imagesection}
\left( \begin{array}{cc}
\cos (t+ k \pi) & \sin (t+ k \pi) \\
-\sin (t+ k \pi) & \cos (t+ k \pi) \end{array} \right) \left( \begin{array}{cc}
u a(u,t+ k \pi)  & u b(u,t+ k \pi) \\
0 & u a^{-1}(u,t+ k \pi) \end{array} \right), \nonumber \end{equation} }
for the continuously differentiable functions $a(1,t)$, $b(1,t)$ the inequalities (\ref{diffinequalitymas}) hold and 
$a(1,k \pi)=1$, $k \in \{0,1\}$, for 
$u=1$ the continuously differentiable functions $\bar{a}_u(t+k \pi)=a(u,0)a^{-1}(u,t+ k \pi)=a(u,0) a^{-1}(u,k \pi) 
a^{-1}(1,t)$,
$\bar{b}_u(t+k \pi)=-b(u,t+k \pi)=-b(u,k \pi) a(1,t)- b(1,t) a^{-1}(u,k \pi)$ satisfy inequalities 
(\ref{diffinequality}). Hence the product
${\mathcal T} {\mathcal K}$ given in the assertion is the image of a ${\cal C}^1$-differentiable section of a multiplicative loop $Q^{\ast }$ of a quasifield $Q$.
\end{proof}

\begin{Prop} \label{compactuj}
The set $\Lambda _{Q^{\ast }}$ of all left translations of the multiplicative loop $Q^{\ast }$ for a locally compact connected topological $2$-dimensional quasifield $Q$ with $(1,0)^t$ as identity of $Q^{\ast }$ contains the group 
${\mathrm{SO}}_2(\mathbb R)$ if and only if
$\Lambda _{Q^{\ast }}$ has the form
\begin{equation} \label{lambda} \Lambda _{Q^{\ast }}=\left\{ \left( \begin{array}{cc}
\cos t & \sin t \\
-\sin t & \cos t \end{array} \right) \left( \begin{array}{cc}
u a(u, 0) & u b(u,0) \\
0 & u  a^{-1}(u, 0) \end{array} \right), u >0, t \in [0, 2 \pi ) \right\}  \end{equation}
where $a(u,0)$, $b(u,0)$ are arbitrary continuous functions with $a(u,0)>0$
such that $u a(u,0)$ is strictly monotone. In this case $Q^{\ast }$ is decomposable. \end{Prop}
\begin{proof}
If the set $\Lambda _{Q^{\ast }}$ contains the group ${\mathrm{SO}}_2(\mathbb R)$, then for each fixed $u >0$ the function $a(u,t)$ is constant with value $1$ and the function $b(u,t)$ is constant with value $0$. So the functions 
$a(u,t)=a(u,0)$, $b(u,t)=b(u,0)$ do not depend on the variable $t$. Hence the identities in Proposition \ref{productloop} are satisfied and the set $\Lambda _{Q^{\ast }}$ has the form ${\mathcal T} {\cal W}$ as in the assertion.

Conversely, if $u a(u,0)$ is a strictly monotone continuous function, then for arbitrary continuous functions 
$a(u,0), b(u,0)$ with $a(u,0)>0$ the set given by (\ref{lambda}) is the set $\Lambda _{Q^{\ast }}$ of all left translations of the multiplicative loop $Q^{\ast }$ of a locally compact quasifield such that $\Lambda _{Q^{\ast }}$ contains the group 
${\mathrm{SO}}_2(\mathbb R)$ and $Q^{\ast }$ has identity $(1,0)^t$.
Hence $Q^{\ast }$ is decomposable and the assertion is proved. \end{proof}

\centerline{\bf 6. Betten's classification of $4$-dimensional translation planes}

\bigskip
\noindent
Using $2$-dimensional spreads, Betten in \cite{betten}, \cite{betten2}, \cite{betten3}, \cite{betten4}, \cite{betten5}, \cite{betten6}, see also \cite{knarr} and \cite{ortleb}, has classified all locally compact $4$-dimensional translation planes which admit an at least $7$-dimensional collineation group. His normalized $2$-dimensional spreads are images of sharply transitive sections $\sigma ' : G/H' \to G$, where $G$ is the connected component of the group ${\mathrm{GL}}_2(\mathbb R)$, $H'$ is the subgroup $\left\{ \left( \begin{array}{cc}
1 & c \\
0 & d \end{array} \right), d>0, c \in \mathbb R \right\}$ (cf. \cite{betten1}, \cite{betten}) and $\sigma '(G/H')$ consists of the elements
\begin{equation} \label{bettensection}
\left( \begin{array}{cc}
\cos t & \sin t \\
-\sin t & \cos t \end{array} \right) \left( \begin{array}{cc}
r a(r, t) & 0 \\
0 & r^{-1} a^{-1}(r, t) \end{array} \right) \left( \begin{array}{cc}
1 & b(r,t) a^{-1}(r,t) \\
0 & r^2 \end{array} \right). \nonumber \end{equation}
With respect to the stabilizer $H=\left\{ \left( \begin{array}{cc}
a & b \\
0 & a^{-1} \end{array} \right), a>0, b \in \mathbb R \right\}$ the sharply transitive section $\sigma '$ transforms into a sharply transitive section $\sigma :G/H \to G$ given by (\ref{section}),
because the elements of $\sigma '(G/H')$ coincide with
\begin{equation} \left( \begin{array}{cc}
\cos t & \sin t \\
-\sin t & \cos t \end{array} \right) \left( \begin{array}{cc}
r  & 0 \\
0 & r \end{array} \right) \left( \begin{array}{cc}
a(r,t) & b(r,t)  \\
0 & a^{-1}(r,t) \end{array} \right). \nonumber \end{equation}

\begin{Prop} \label{prop1}
Let ${\cal A}$ be a $4$-dimensional non-desarguesian translation plane admitting an $8$-dimensional collineation group such that ${\cal A}$ is coordinatized by the locally compact topological quasifield $Q$. Then the multiplicative loop 
$Q^{\ast }$ can be given by one of the following sets $\Lambda _{Q^{\ast }}$ of the left translations of $Q^{\ast }$:

\smallskip
\noindent
a) $\Lambda _{Q^{\ast }}$ has the form (\ref{diagonal}) with $a(1, t)=1$ and $b(1, t)=0$ for $0 \le t \le \pi $, \ 
$a(1,t)=1/\sqrt{\cos^2 t+ \frac{ \sin^2 t}{w}}$ and
$b(1,t)=a(1,t) \frac{1-w}{w} \sin t \cos t$ for $\pi < t < 2 \pi $. The quasifields $Q_w$, $w>1$, correspond to a 
one-parameter family of planes ${\cal A}_w$.

\smallskip
\noindent
b) $\Lambda _{Q^{\ast }}$ is the range of the section given by (\ref{section}) such that for 
$\alpha \ge \frac{-3 \beta ^2}{4}$ one has
$a(r, t)= \frac{\sqrt{\alpha ^2 + \beta ^2}}{\alpha + \beta ^2}$ and 
$b(r,t)= \frac{ \beta (- \alpha +1)}{\sqrt{\alpha ^2 + \beta ^2}}$ with
$r \cos(t)=\alpha \frac{\alpha + \beta ^2}{\sqrt{\alpha ^2 + \beta ^2}}$, 
$r \sin(t)= -\beta \frac{\alpha + \beta ^2}{\sqrt{\alpha ^2 + \beta ^2}}$.
\newline
\noindent
For $\alpha < \frac{-3 \beta ^2}{4}$ we have $a(r,t)=3 \sqrt{\frac{\alpha ^2 + \beta ^2}{\alpha ^2}}$ and 
$b(r,t)= \beta \sqrt{\frac{\alpha ^2 + \beta ^2}{\alpha ^2}}+ \frac{\beta \alpha}{3 \sqrt{\alpha ^2( \alpha ^2 +\beta ^2)}}$
with 
$r \cos(t)=\frac{\alpha }{3} \sqrt{\frac{\alpha ^2}{\alpha ^2 + \beta ^2}}$, 
$r \sin(t)=-\frac{\beta }{3} \sqrt{\frac{\alpha ^2}{\alpha ^2 + \beta ^2}}$. The quasifield $Q$ coordinatizes a 
single plane.

\smallskip
\noindent
c) $\Lambda _{Q^{\ast }}$ is the range of the section given by (\ref{section}) such that
$a(r,t)= \sqrt{\frac{v^2+ s^2}{\frac{s^4}{3}+s^2 v + v^2}}$,
$b(r,t)=\frac{-\frac{s^3 v}{3}+s^3+s v}{\sqrt{(\frac{s^4}{3}+s^2 v + v^2)(s^2+ v^2)}}$ with
{\tiny \begin{equation} r \cos(t)=v \sqrt{\frac{\frac{s^4}{3}+s^2 v + v^2}{s^2+v^2}}, 
\ r \sin(t)=-s \sqrt{\frac{\frac{s^4}{3}+s^2 v + v^2}{s^2+v^2}}.
\nonumber \end{equation} }
The quasifield $Q$ coordinatizes a single plane.

\smallskip
\noindent
In case a) the multiplicative loop $Q^{\ast }_w$ is decomposable and a split extension of the normal subgroup 
$\widetilde{N^{\ast }} \cong \mathbb R$ corresponding to the connected component of 
$\widetilde{{\mathcal K}}=\left\{ \left(\begin{array}{cc}
r & 0 \\
0 & r \end{array} \right), 0 \neq r \in \mathbb R \right\}$
with a subloop homeomorphic to the $1$-sphere.
In case c) the set ${\mathcal K}$ of the left translations of $Q^{\ast }$ corresponding to the kernel $K_r$ of the quasifield $Q$ has the form $\widetilde{{\mathcal K}}$. In case b) the set ${\mathcal K}$ has the form 
\begin{equation} \left\{ \left(\begin{array}{cc}
\alpha & 0 \\
0 & \alpha \end{array} \right), \alpha >0 \right\} \cup \left\{ \left(\begin{array}{cc}
\alpha & 0 \\
0 & \frac{\alpha }{9} \end{array} \right), \alpha <0 \right\}. \nonumber \end{equation}
In cases b) and c) the multiplicative loops $Q^{\ast }$ are not decomposable and quasi-simple.
\end{Prop}
\begin{proof} If the translation complement of ${\cal A}$ is the group ${\mathrm{GL}}_2(\mathbb R)$ and acts reducibly on
$\mathbb R^4$, then one obtains the  one-parameter family ${\cal A}_w$, $w>1$, of the non-desarguesian translation planes corresponding to the following spreads:
\begin{equation} \label{spread1}
\{ S \} \cup \left\{ \left( \begin{array}{cc}
s & -v \\
v & s \end{array} \right), s, v \in \mathbb R, v \ge 0 \right\} \cup \left\{ \left( \begin{array}{cc}
s & \frac{-v}{w} \\
v & s \end{array} \right), s, v \in \mathbb R, v < 0 \right\},  \nonumber \end{equation}
$w > 1$ (cf. \cite{betten}, Satz 5, p. 144).
Any such spread coincides with the set $\Lambda_{Q^{\ast}}$ in (\ref{diagonal}) with $a(1,t)$ and $b(1,t)$ as in  assertion a). 
By Theorem \ref{normal}
the multiplicative loop $Q^{\ast }_w$ is a split extension of a normal subgroup $N^{\ast }$ with a $1$-dimensional compact loop.
By Theorem \ref{quasisimple} the loop $Q^{\ast }_w$ is decomposable. 
Hence $\widetilde{N^{\ast }}$ has the form as in the assertion.

\noindent
If the translation complement ${\mathrm{GL}}_2(\mathbb R)$ acts irreducibly on $\mathbb R^4$, then one obtains a single plane ${\cal A}$ generated by the spread
{\tiny \begin{equation} \label{spread2}
\{ S \} \cup \left\{ \left( \begin{array}{cc}
\alpha  & -\alpha \beta -\beta^3  \\
\beta  & \alpha + \beta^2 \end{array} \right), \alpha, \beta \in \mathbb R, \alpha \ge \frac{-3 \beta^2}{4} \right\} \cup 
\left\{ \left( \begin{array}{cc}
\alpha  & \frac{1}{3} \alpha \beta \\
\beta & \frac{\alpha }{9}+ \frac{\beta ^2}{3} \end{array} \right), \alpha, \beta \in \mathbb R, \alpha < \frac{-3 \beta ^2}{4} \right\},  \end{equation} }
(cf. \cite{betten3}, Satz, p. 553).

\noindent
If the translation complement is solvable, then one gets a single plane ${\cal A}$ generated by the spread
\begin{equation} \label{spread3}
\{ S \} \cup \left\{ \left( \begin{array}{cc}
v  & -\frac{s^3}{3} \\
s      & s^2+ v \end{array} \right), s, v \in \mathbb R \right\},   \end{equation}
(cf.  \cite{betten2},  Satz  2 (b), p. 331).

\noindent
The spread (\ref{spread2}), respectively (\ref{spread3}) coincides with the image of the section $\sigma $ in 
(\ref{section}) with the well defined functions $a(r,t)$ and $b(r,t)$ given in assertion b), respectively c). In case c) one has $a(r,k \pi )=1$, $b(r,k \pi )=0$, for all $r>0$, $k \in \{0, 1\}$ hence Remark \ref{diagonaluj} gives the form 
${\mathcal K}$ of the assertion. In case b) we have $a(r,0)=1$, $a(r, \pi)=3$, $b(r,l \pi )=0$,  for all $r>0$, 
$l \in \{0, 1 \}$. 
These give the form ${\mathcal K}$ of the assertion.

\noindent
For  decomposable $Q^{\ast }$, the identity $a(r,t+ k \pi)=a(1,t) a(r,k \pi)$ holds for all
$r > 0$, $t \in [0, 2 \pi)$, $k \in \{0, 1\}$ (cf. Proposition \ref{productloop}). In case b) for
$-3 \le \alpha \le 1$ one has $a(1,t)=\sqrt{\alpha ^2 - \alpha +1}$
which yields a contradiction. In case c) we have $a(r, \frac{\pi}{4}+ k \pi)= \sqrt{ \frac{2}{1-s + \frac{s^2}{3}}}$,
$s \in \mathbb R \setminus \{ 0 \}$ and the condition $a(r, \frac{\pi}{4}+ k \pi)= a(1, \frac{\pi}{4})$ gives a contradiction.
Hence in both cases $Q^{\ast }$ is not decomposable and therefore quasi-simple (cf. Theorem \ref{quasisimple}).
\end{proof}

\medskip
\noindent
If the translation complement of a $4$-dimensional topological plane ${\cal A}$ has dimension $3$, then the point 
$\infty $ of the line
$S=\{(0,0,u,v), u,v \in \mathbb R\}$ is fixed under the seven-dimensional collineation group $\Gamma $ of ${\cal A}$.

\begin{Prop} \label{prop2}
Let $Q$ be a $2$-dimensional quasifield coordinatizing a $4$-dimensional locally compact translation plane ${\cal A}$ such that the $7$-dimensional collineation group $\Gamma $ of ${\cal A}$ acts transitively on the points of
$W \setminus \{ \infty \}$, where $W$ is the translation axis of ${\cal A}$ and the kernel of the action of the translation complement on the line $S$ has dimension $1$.  Then the multiplicative loop $Q^{\ast }$  can be given by one of the following sets $\Lambda _{Q^{\ast }}$ of the left translations of $Q^{\ast }$:

\smallskip
\noindent
a) $\Lambda _{Q^{\ast }}$ is the range of the section (\ref{section}) such that
$$a(r,t)= \sqrt{\frac{s^2+v^2}{s^2 v+ v^2+ \frac{s^4}{3}+ s^2}} \ \hbox{and} \
b(r,t)=\frac{s^3-\frac{s^3v}{3}}{\sqrt{(s^2 v+ v^2+ \frac{s^4}{3}+ s^2)(s^2+v^2)}}$$
with $r \cos(t)=v \sqrt{\frac{s^2 v+ v^2+ \frac{s^4}{3}+ s^2}{s^2+v^2}}$, $r \sin(t)=-s \sqrt{\frac{s^2 v+ v^2+ \frac{s^4}{3}+ s^2}{s^2+v^2}}$.
The quasifield $Q$ corresponds to a single plane.

\smallskip
\noindent
b) $\Lambda _{Q^{\ast }}$ is the range of the section given by (\ref{section}) such that
{\tiny \begin{equation} a(r,t)= \sqrt{\frac{v^2+u^2+2 \gamma ^2(1-\cos(u))-2 v \gamma \sin(u)-2 \gamma u \cos(u)+ 2 \gamma u}{v^2+ u^2 - 2 \gamma ^2 + 2 \gamma ^2 \cos(u) }} \ \hbox{and} \nonumber \end{equation}
\begin{equation} b(r,t)=\frac{-2 u \gamma \sin u+ 2 v \gamma \cos u- 2 v \gamma}{\sqrt{v^2+u^2+2 \gamma ^2(1-\cos u)-2 v \gamma \sin u-2 \gamma u \cos u+ 2 \gamma u} \sqrt{v^2+u^2-2 \gamma ^2(1-\cos u)}} \nonumber \end{equation} }
with
{\tiny \begin{equation} r \cos(t)=(v- \gamma \sin (u))  \sqrt{\frac{v^2+ u^2 - 2 \gamma ^2 + 2 \gamma ^2 \cos u}{v^2+u^2+2 \gamma ^2(1-\cos(u))-2 v \gamma \sin(u)-2 \gamma u \cos(u)+ 2 \gamma u}}, \nonumber \end{equation}
\begin{equation} r \sin(t)=(u- \gamma (\cos(u)-1)) \sqrt{\frac{v^2+ u^2 - 2 \gamma ^2 + 2 \gamma ^2 \cos u}{v^2+u^2+2 \gamma ^2(1-\cos(u))-2 v \gamma \sin(u)-2 \gamma u \cos(u)+ 2 \gamma u}}. \nonumber \end{equation} } 
\noindent
The quasifields
$Q_{\gamma }$ coordinatize a one-parameter family of planes
${\cal A}_{\gamma}, 0 < |\gamma | \le 1$.

\smallskip 
\noindent
In all cases the multiplicative loop $Q^{\ast }$ is not decomposable and quasi-simple. The set ${\mathcal K}$ of the left translations of $Q^{\ast }$ corresponding to the kernel of the quasifield $Q$ has the form 
$\left\{ \left( \begin{array}{cc}
r & 0 \\
0 & r \end{array} \right), 0 \neq r \in \mathbb R \right\}$.
\end{Prop}
\begin{proof}
If the translation complement $C$ leaves a $1$-dimensional subspace of $S$ invariant, then one obtains a single plane 
${\cal A}$ which corresponds to the following spread:
\begin{equation} \label{spread4}
\{ S \} \cup \left\{ \left( \begin{array}{cc}
v  & -\frac{s^3}{3}- s \\
s  & s^2+ v  \end{array} \right), s, v \in \mathbb R \right\} \end{equation}
(cf.  \cite{salzmann}, \ 73.10.,  \cite{betten2}, pp. 330-331).

\noindent
If the translation complement acts transitively on the $1$-dimensional subspaces of $S$, then one gets a
one-parameter family $E_{\gamma}, 0 < |\gamma | \le 1$, of planes which are  generated by the normalized spread
\begin{equation} \label{spreadgamma}
\{ S \} \cup \left\{ \left( \begin{array}{cc}
v- \gamma \sin u & u+ \gamma (\cos u-1) \\
\gamma (\cos u-1)- u & v+ \gamma \sin u \end{array} \right), u, v \in \mathbb R \right\},  \end{equation}
(\cite{betten6}, Satz, p. 128, \  \cite{knarr}, Proposition  5.8).
The spread (\ref{spread4}), respectively (\ref{spreadgamma}) coincides with the image of the section $\sigma $ in 
(\ref{section})  such that the well defined functions $a(r,t)$ and $b(r,t)$ are given in assertion a), respectively b).
Since in both cases  one has
$a(r,k \pi)=1, b(r,k \pi)=0$, for all $r>0$, $k \in \{0, 1\}$, Remark \ref{diagonaluj} gives the form of ${\mathcal K}$.
Moreover, in case a) one has
$a(r, \frac{\pi}{4}+ k \pi)= \sqrt{ \frac{2}{2+ v + \frac{v^2}{3}}},\ v \in \mathbb R \setminus \{ 0 \}$. In case b) for 
$v=1$ we get
{\tiny \begin{equation} \label{ujequuj}
a(r_j,t_j)= \sqrt{\frac{1+u^2+2 \gamma ^2(1-\cos u)-2 \gamma \sin u-2 \gamma u \cos u+ 2 \gamma u}{1 + u^2 - 2 \gamma ^2 + 2 \gamma ^2 \cos u }}, \ \ \ a(1,t_j)=1. \nonumber \end{equation} }
For decomposable $Q^{\ast }$ one has
$a(r, t+ k \pi)= a(1, t) a(r, k \pi)$ for all $r \in \mathbb R \setminus \{ 0 \}$, $t \in [0, 2 \pi)$, $k \in \{0, 1\}$  
(cf. Proposition \ref{productloop}) which yields a contradiction.
Thus in both cases $Q^{\ast }$ is not decomposable and hence quasi-simple (cf. Theorem \ref{quasisimple}).
\end{proof}

\begin{Prop} \label{prop3}
Let $Q$ be a $2$-dimensional quasifield coordinatizing a $4$-dimensional locally compact translation plane ${\cal A}$
such that the translation complement $C$ of the $7$-dimensional collineation group $\Gamma $ of ${\cal A}$ has an orbit of dimension $1$ on
$W \setminus \{ 0 \}$, $C$ leaves in the set of lines through the origin only $S$ fixed and the kernel of its action on 
$S$ has positive dimension. Then the multiplicative loop $Q^{\ast }$ can be given by one of the following sets 
 $\Lambda _{Q^{\ast }}$ of the left translations of $Q^{\ast }$:

\smallskip
\noindent
a) $\Lambda _{Q^{\ast }}$ is the range of the section (\ref{section}) such that for $\beta \ge 0$ one has
{\tiny \begin{equation}
a(r,t)= \sqrt{\frac{\alpha ^2+ \beta^2}{\alpha ^2 + z \alpha \beta^{\frac{1}{1+s}}- w \beta^{\frac{2}{1+s}}}} \ \hbox{and} \
b(r,t)=\frac{w \alpha \beta ^{\frac{1-s}{1+s}}+ \alpha \beta + z \beta ^{\frac{2+s}{1+s}}}{\sqrt{\alpha ^2+ \beta^2}
\sqrt{\alpha ^2 + z \alpha \beta^{\frac{1}{1+s}}- w \beta^{\frac{2}{1+s}}}} \nonumber \end{equation} }
{\tiny \begin{equation} \hbox{with} \  r \cos(t)=\alpha \sqrt{\frac{\alpha ^2 + z \alpha \beta^{\frac{1}{1+s}}- 
w \beta^{\frac{2}{1+s}}}{\alpha ^2+ \beta^2}}, \
r \sin(t)=- \beta \sqrt{\frac{\alpha ^2 + z \alpha \beta^{\frac{1}{1+s}}- w \beta^{\frac{2}{1+s}}}{\alpha ^2+ \beta^2}}. \nonumber \end{equation}}
\newline
\noindent
For $\beta <0$ one gets
{\tiny \begin{equation}
a(r',t)= \sqrt{\frac{\alpha ^2+ \beta^2}{\alpha ^2 + q \alpha (-\beta)^{\frac{1}{1+s}}+ p (-\beta)^{\frac{2}{1+s}}}} \ \hbox{and} \
b(r',t)=\frac{p \alpha (-\beta )^{\frac{1-s}{1+s}}+ \alpha \beta -q (-\beta )^{\frac{2+s}{1+s}}}{\sqrt{\alpha ^2+ \beta^2} \sqrt{\alpha ^2 + q \alpha (-\beta )^{\frac{1}{1+s}}+p (-\beta )^{\frac{2}{1+s}}}} \nonumber \end{equation} }
with
{\tiny \begin{equation}
r' \cos(t)=\alpha \sqrt{\frac{\alpha ^2 + q \alpha (-\beta )^{\frac{1}{1+s}}+p (-\beta )^{\frac{2}{1+s}}}{\alpha ^2+ \beta^2}} \ \hbox{and} \
r' \sin(t)=- \beta \sqrt{\frac{\alpha ^2 + q \alpha (-\beta )^{\frac{1}{1+s}}+p (-\beta )^{\frac{2}{1+s}}}{\alpha ^2+ \beta^2}}. \nonumber \end{equation} }

\noindent
The quasifields $Q_{s,w,z,p,q}$ coordinatize a family of planes
${\cal A}_{s,w,z,p,q}$ such that the parameters $s,w,z,p,q$ satisfy the conditions
$0 < s <1$, $z^2+ 4 w(1-s^2) \le 0$, $q^2-4 p(1-s^2) \le 0$.

\smallskip
\noindent
b) $\Lambda _{Q^{\ast }}$ is the range of the section (\ref{section}) such that for $\beta \ge 0$ we have
{\tiny \begin{equation} a(r,t)= \sqrt{\frac{\alpha ^2+ \beta^2}{\alpha ^2 + z \alpha \beta - w \beta^2 + 2 \alpha \beta \ln \beta + z \beta ^2 \ln \beta +
\beta ^2 (\ln \beta )^2}} \ \hbox{and} \nonumber \end{equation}
\begin{equation}  b(r,t)=\frac{(w+1) \alpha \beta + z \beta ^2 - z \alpha \beta \ln \beta - \alpha \beta (ln \beta )^2 + 2 \beta ^2 \ln \beta}{\sqrt{\alpha ^2+ \beta^2} \sqrt{\alpha ^2 + z \alpha \beta + 2 \alpha \beta \ln \beta - w \beta ^2+ z \beta ^2 \ln \beta + \beta ^2 (\ln \beta )^2}} \nonumber \end{equation} }
with
{\tiny \begin{equation}
r \cos(t)= \alpha \sqrt{\frac{\alpha ^2 + z \alpha \beta - w \beta^2 + 2 \alpha \beta \ln \beta + z \beta ^2 \ln \beta +
\beta ^2 (\ln \beta )^2}{\alpha ^2+ \beta^2}}, \nonumber \end{equation}
\begin{equation}
r \sin(t)= -\beta \sqrt{\frac{\alpha ^2 + z \alpha \beta - w \beta^2 + 2 \alpha \beta \ln \beta + z \beta ^2 \ln \beta +
\beta ^2 (\ln \beta )^2}{\alpha ^2+ \beta^2}}. \nonumber \end{equation} }
For $\beta <0$ we obtain
{\tiny \begin{equation} a(r',t)= \sqrt{\frac{\alpha ^2+ \beta^2}{\alpha ^2 - q \alpha \beta + p \beta ^2 + (2 \alpha \beta - q \beta ^2) \ln (- \beta ) +
\beta ^2 (\ln (- \beta ))^2}}  \ \hbox{and} \nonumber \end{equation}
\begin{equation} b(r',t)=\frac{(1-p) \alpha \beta - q \beta ^2+ (2 \beta ^2+ q \alpha \beta ) \ln (- \beta ) - \alpha \beta (\ln (- \beta ))^2}{\sqrt{\alpha ^2+ \beta^2} \sqrt{\alpha ^2 - q \alpha \beta + p \beta ^2 + (2 \alpha \beta - q \beta ^2) \ln (- \beta ) + \beta ^2 (\ln (- \beta ))^2}} \nonumber \end{equation} }
with
{\tiny \begin{equation}
r' \cos(t)= \alpha \sqrt{\frac{\alpha ^2 - q \alpha \beta + p \beta ^2 + (2 \alpha \beta - q \beta ^2) \ln (- \beta ) +
\beta ^2 (\ln (- \beta ))^2}{\alpha ^2+ \beta^2}}, \nonumber \end{equation}
\begin{equation}
r' \sin(t)= -\beta \sqrt{\frac{\alpha ^2 - q \alpha \beta + p \beta ^2 + (2 \alpha \beta - q \beta ^2) \ln (- \beta ) +
\beta ^2 (\ln (- \beta ))^2}{\alpha ^2+ \beta^2}}. \nonumber \end{equation} }
The quasifields $Q_{w,z,p,q}$ coordinatize a family of planes
${\cal A}_{w,z,p,q}$ such that for the parameters $w,z,p,q$ the relations $\left( \frac{z}{2} \right)^2 \le -w-1$, $\left( \frac{q}{2} \right)^2 \le p-1$ hold.

\smallskip
\noindent
c) $\Lambda _{Q^{\ast }}$ is the range of the section given by (\ref{section}) such that
$a(r,k \pi)= 1$ and $b(r,k \pi)=0$, $k \in \{0, 1\}$ with $r=|\beta |$. 
\newline
\noindent
For $u \in \mathbb R$, $\beta >0$, we get
{\tiny \begin{equation}
a(r,t)=\sqrt{\frac{u^2+ \sin ^2(l)(w^2+2 z u+ z^2) + \cos^2(l)- (2 u w+ 2 u+ 2 z) \sin (l) \cos (l)}{u^2+ u z -w}}, \nonumber \end{equation}
\begin{equation}
b(r,t)= \frac{\cos^2(l)(2 u w+2 u + 2 z)+ \sin(l) \cos(l)(1-w^2-z^2-2 u z)-(u+z+u w)}{\sqrt{(u^2+ \sin ^2(l)(w^2+2 z u+ z^2) + \cos^2(l)- (2 u w+ 2 u+ 2 z) \sin (l) \cos (l))(u^2+ uz -w)}}  \nonumber \end{equation} }
with
{\tiny \begin{equation}
r \cos(t)= \beta \left(u- (w+1) \sin(l) \cos(l) + z \sin ^2(l) \right) \sqrt{\frac{u^2+ uz -w}{u^2+ \sin ^2(l)(w^2+2 z u+ z^2) + \cos^2(l)- (2 u w+ 2 u+ 2 z) \sin (l) \cos (l)}}, \nonumber \end{equation}
\begin{equation}
r \sin(t)= \beta \left(w \sin^2(l)+ z \sin(l) \cos(l)- \cos^2(l) \right) \sqrt{\frac{u^2+ u z -w}{u^2+ \sin ^2(l)(w^2+2 z u+ z^2) + \cos^2(l)- (2 u w+ 2 u+ 2 z) \sin (l) \cos (l)}}, \nonumber \end{equation} }
where
$l= \frac{1}{k} \ln \beta $.
For $u \in \mathbb R$, $\beta <0$ one obtains 
{\tiny \begin{equation}
a(r',t')= \sqrt{\frac{u^2+ \sin ^2(l_1)(q^2+2 q u+ p^2) + \cos^2(l_1)+ ( 2 u+ 2 q-2 u p) \sin (l_1) \cos (l_1)}{u^2+ u q +p}} \nonumber \end{equation}
\begin{equation}
b(r',t')= \frac{ \sin(l_1) \cos(l_1)(1-2 u q-p^2-q^2)+ \sin^2(l_1) (2 q+ 2 u-2 u p)+(u p-q-u)}{\sqrt{(u^2+ \sin ^2(l_1)(q^2+2 q u+ p^2) + \cos^2(l_1)+ ( 2 q+ 2 u- 2 u p) \sin (l_1) \cos (l_1))(u^2+ u q +p)}} \nonumber \end{equation} }
with
{\tiny \begin{equation}
r' \cos(t')= \beta \left((p-1) \sin(l_1) \cos(l_1)-q \sin^2(l_1)-u \right) \cdot  \nonumber \end{equation}
\begin{equation} \sqrt{\frac{u^2+ uq +p}{u^2+ \sin ^2(l_1)(q^2+2 q u+ p^2) + \cos^2(l_1)+ ( 2 u+ 2 q-2 u p) \sin (l_1) \cos (l_1)}}, \nonumber \end{equation}
\begin{equation}
r' \sin(t')= -\beta \left(\cos^2(l_1)+ q \sin(l_1) \cos(l_1) + p \sin^2(l_1) \right) \cdot \nonumber \end{equation}
\begin{equation} \sqrt{\frac{u^2+ u q +p}{u^2+ \sin ^2(l_1)(q^2+2 q u+ p^2) + \cos^2(l_1)+ ( 2 u+ 2 q-2 u p) \sin (l_1) \cos (l_1)}}, \nonumber \end{equation}}
where
$l_1= \frac{1}{k} \ln (-\beta )$.
\newline
\noindent
The quasifields $Q_{k,w,z,p,q}$ coordinatize a family of planes
${\cal A}_{k,w,z,p,q}$ such that for the parameters $k,w,z,p,q$ one has
$k \neq 0$, $(4+k^2)(z^2+(w+1)^2) \le k^2(1-w)^2$, $(4+k^2)(q^2+(p-1)^2) \le k^2(p+1)^2$, $(w,z,p,q) \neq (-1,0,1,0)$.
\newline
\noindent
In all cases $Q^{\ast }$ is not decomposable and quasi-simple. The set of the left translations of $Q^{\ast }$ belonging to the kernel of $Q$ is
${\mathcal K}=\left\{ \left( \begin{array}{cc}
r & 0 \\
0 & r \end{array} \right), 0 \neq r \in \mathbb R \right\}$.
\end{Prop}
\begin{proof}
If the translation complement $C$ fixes two $1$-dimensional subspaces of $S$, then we have a family of
translation planes ${\mathcal A}_{s,w,z,p,q}$ such that the normalized spreads belonging to these planes are given as follows:
{\tiny \begin{equation} \label{spread5} \{S \} \cup \left\{ \left( \begin{array}{cc}
\alpha  & w \beta^{\frac{1-s}{1+s}} \\
\beta & z \beta^{\frac{1}{1+s}}+ \alpha \end{array} \right), \alpha \in \mathbb R, \beta \ge 0 \right\} \cup
\left\{ \left( \begin{array}{cc}
\alpha  & p (-\beta)^{\frac{1-s}{1+s}} \\
\beta & q (-\beta)^{\frac{1}{1+s}}+ \alpha \end{array} \right), \alpha \in \mathbb R, \beta < 0 \right\},  \end{equation} }
(cf. \cite{betten4},  Satz 1, pp. 411-412).

\noindent
If the translation complement $C$ fixes only one $1$-dimensional subspace of $S$, then there is a family of translation planes
${\mathcal A}_{w,z,p,q}$ such that the corresponding normalized spreads have the form:
{\tiny \begin{equation} \label{spread6} \{S \} \cup \left\{ \left( \begin{array}{cc}
\alpha  & w \beta -z \beta \ln \beta - \beta (\ln \beta)^2  \\
\beta &  \alpha + z \beta + 2 \beta \ln \beta \end{array} \right), \alpha \in \mathbb R, \beta \ge 0 \right\} \cup
\left\{ \left( \begin{array}{cc}
\alpha  & -p \beta - \beta (\ln (- \beta ))^2 + q \beta \ln (- \beta)  \\
\beta & q (-\beta) + \alpha + 2 \beta \ln(- \beta ) \end{array} \right), \alpha \in \mathbb R, \beta < 0 \right\}  
\end{equation} }
(cf. Satz 2, \cite{betten4}, pp. 418-419).

\noindent
If the translation complement $C$ acts transitively on the $1$-dimensional subspaces of $S$, then we have a family of
translation planes ${\mathcal A}_{k,w,z,p,q}$ such that the normalized spreads belonging to these planes have the form
{\tiny \begin{equation} \label{spreadSatz3}
\{S \} \cup \left\{ \beta \left( \begin{array}{cc}
1 &  0 \\
0 & 1 \end{array} \right),  \beta  \in \mathbb R \right\} \cup \nonumber \end{equation}
\begin{equation} \left\{ \beta \left( \begin{array}{cc}
u-(w+1) \sin(l) \cos(l)+ z \sin^2(l) & w \cos^2(l)-z \sin (l) \cos(l) - \sin ^2(l)  \\
\cos^2(l)-z \sin (l) \cos(l) - w \sin ^2(l) & z \cos^2(l)+(w+1) \sin (l) \cos(l) +u \end{array} \right), u \in \mathbb R, \beta >0 \right\} \cup \nonumber \end{equation}
\begin{equation} \left\{ \beta  \left( \begin{array}{cc}
(p-1) \sin(l_1) \cos(l_1)-q \sin^2(l_1)-u & q \sin (l_1) \cos(l_1) -p \cos^2(l_1)- \sin ^2(l_1) \\
\cos^2(l_1)+ q \sin (l_1) \cos(l_1) +p \sin ^2(l_1) & (1-p) \sin (l_1) \cos(l_1)-q \cos^2(l_1)-u \end{array} \right), 
u \in \mathbb R, \beta < 0 \right\},  \end{equation} }
\noindent
where $l= \frac{1}{k} \ln \beta $, $l_1= \frac{1}{k} \ln (-\beta)$  (cf. \cite{ortleb}, Proposition 4.1, p. 6, and 
\cite{betten4}, Satz 3, pp. 422-423).
\noindent
The spreads (\ref{spread5}), respectively (\ref{spread6}), respectively (\ref{spreadSatz3}) coincide with the image of the section $\sigma $ in (\ref{section}) such that the well defined functions $a(r,t)$ and $b(r,t)$ are given in assertion a), respectively b), respectively c). Since in all three cases we have $a(r,k \pi)=1$, $b(r,k \pi)=0$, $r >0$, 
$k \in \{0, 1\}$, Remark 
\ref{diagonaluj} shows that ${\mathcal K}$ has the form as in the assertion.
In case a), respectively b) for $\beta >0$ one gets
$a(r,\frac{\pi}{4}+ \pi)= \frac{\sqrt{2} \beta }{\sqrt{\beta ^2- z \beta ^{\frac{2+s}{1+s}} 
-w \beta ^{\frac{2}{1+s}}}}$, respectively
$a(r, \frac{\pi}{4}+ \pi)= \sqrt{\frac{2}{1-z-w-2 \ln \beta + z \ln \beta +(\ln \beta )^2}}$.
In case c) for $u=0$, $\beta >0$ we get that $a(1,t_j)$ is constant. These relations give a contradiction to the condition 
$a(r,t+ k \pi)=a(1,t)$, $r>0$, $t \in [0, 2 \pi)$, $k \in \{ 0, 1\}$ 
of Proposition \ref{productloop}. Hence in all cases $Q^{\ast }$ is not decomposable  and quasi-simple (cf. Theorem 
\ref{quasisimple}).
\end{proof}

\begin{Prop} \label{prop4}
Let $Q$ be a $2$-dimensional quasifield coordinatizing a $4$-dimensional locally compact translation plane ${\cal A}$
such that the translation complement $C$ of the $7$-dimensional  collineation group $\Gamma $ of ${\cal A}$ has an orbit of  dimension 
$1$ on
$W \setminus \{ 0 \}$, $C$ leaves only $S$ in the set of lines through the origin  fixed and the kernel of its action on 
$S$ is  zero-dimensional. Then the set $\Lambda _{Q^{\ast }}$ of all left translations of the
multiplicative loop $Q^{\ast }$ is given by the range of the section (\ref{section}) defined as follows:
For $\alpha \ge - \frac{\beta^2}{2}$ one has
{\tiny \begin{equation} \label{artSatz5elso}
a(r,t)= \sqrt{\frac{\alpha ^2+ \beta^2}{\frac{\alpha \beta^2}{2q}+\frac{\beta ^4}{3q}+\left( \alpha +\frac{\beta ^2}{2} \right)
\left(\alpha + \frac{q-1}{q} \beta ^2 \right)- \frac{p \beta }{q} \left( \alpha +\frac{\beta ^2}{2} \right)^{\frac{3}{2}}}}, \nonumber \end{equation}
\begin{equation} b(r,t)= \frac{\frac{p}{q} \alpha \left( \alpha + \frac{\beta ^2}{2} \right)^{\frac{3}{2}} - \frac{p}{q} \left(\alpha ^2 + \beta^2 \right) + \frac{1-q}{q} \beta \alpha ^2 + \frac{\alpha \beta ^3}{6 q}- \frac{\beta ^3 \alpha }{2}+ \frac{\beta ^3}{2q}+
\frac{\beta ^3}{2} + \alpha \beta }{\sqrt{\alpha ^2+ \beta^2} \sqrt{\frac{\alpha \beta^2}{2q}+\frac{\beta ^4}{3q}+\left( \alpha +\frac{\beta ^2}{2} \right) \left(\alpha + \frac{(q-1)}{q} \beta ^2 \right)- \frac{p \beta }{q} \left( \alpha +\frac{\beta ^2}{2} \right)^{\frac{3}{2}}}},  \nonumber \end{equation} }
with
{\tiny \begin{equation}
r \cos (t)= \alpha \sqrt{\frac{\frac{\alpha \beta^2}{2q}+\frac{\beta ^4}{3q}+\left( \alpha +\frac{\beta ^2}{2} \right) \left(\alpha + \frac{q-1}{q} \beta ^2 \right)- \frac{p \beta }{q} \left( \alpha +\frac{\beta ^2}{2} \right)^{\frac{3}{2}}}{\alpha ^2+ \beta ^2}}, \nonumber \end{equation}
\begin{equation}
r \sin (t)= -\beta \sqrt{\frac{\frac{\alpha \beta^2}{2q}+\frac{\beta ^4}{3q}+\left( \alpha +\frac{\beta ^2}{2} \right) \left(\alpha + \frac{q-1}{q} \beta ^2 \right)- \frac{p \beta }{q} \left( \alpha +\frac{\beta ^2}{2} \right)^{\frac{3}{2}}}{\alpha ^2+ \beta ^2}}. \nonumber \end{equation} }
\noindent 
For $\alpha < - \frac{\beta^2}{2} $  we get
{\tiny \begin{equation} \label{artSatz5}
a(r,t)= \sqrt{\frac{\alpha ^2+ \beta^2}{\frac{\alpha \beta^2}{2q} +\frac{ \beta ^4}{3q}- \left( \alpha +\frac{\beta ^2}{2} \right) \left( \frac{ \alpha z}{q}+ \frac{(z+1) \beta ^2}{q} \right)- \frac{w \beta }{q} \left( -\alpha -\frac{\beta ^2}{2} \right)^{\frac{3}{2}}}},  \nonumber \end{equation}
\begin{equation} b(r,t)= \frac{ \frac{w}{q} \alpha \left( -\alpha - \frac{\beta ^2}{2} \right)^{\frac{3}{2}}+\frac{p }{q} \left(- \alpha ^2-\beta ^2 \right)+
\left(\frac{z+1}{q} \alpha \beta - \frac{z \beta }{q} \right) \left( \alpha +\frac{\beta ^2}{2} \right)-\frac{\alpha \beta ^3}{3q}
+\frac{\beta ^3}{2q}}{\sqrt{\alpha ^2+ \beta^2}{\sqrt{\frac{\alpha \beta^2}{2q} +\frac{ \beta ^4}{3q}- \left( \alpha +\frac{\beta ^2}{2} \right) \left( \frac{ \alpha z}{q}+ \frac{(z+1) \beta ^2}{q} \right)- \frac{w \beta }{q} \left( -\alpha -\frac{\beta ^2}{2} \right)^{\frac{3}{2}}}}}, \nonumber \end{equation} }
with
{\tiny \begin{equation}
r \cos(t)= \alpha \sqrt{\frac{\frac{\alpha \beta^2}{2q} +\frac{ \beta ^4}{3q}- \left( \alpha +\frac{\beta ^2}{2} \right) \left( \frac{ \alpha z}{q}+ \frac{(z+1) \beta ^2}{q} \right)- \frac{w \beta }{q} \left( -\alpha -\frac{\beta ^2}{2} \right)^{\frac{3}{2}}}{\alpha ^2+ \beta^2}}, \nonumber \end{equation}
\begin{equation}
r \sin(t)= -\beta \sqrt{\frac{\frac{\alpha \beta^2}{2q} +\frac{ \beta ^4}{3q}- \left( \alpha +\frac{\beta ^2}{2} \right) \left( \frac{ \alpha z}{q}+ \frac{(z+1) \beta ^2}{q} \right)- \frac{w \beta }{q} \left( -\alpha -\frac{\beta ^2}{2} \right)^{\frac{3}{2}}}{\alpha ^2+ \beta^2}}. \nonumber \end{equation} }
\noindent
The quasifields $Q_{w,z,p,q}$ coordinatize a family of planes
${\cal A}_{w,z,p,q}$ such that the parameters $w,z,p,q$ satisfy $(3w)^2 \le -16 z(z+1)$, $(3p)^2 \le 16 q(q-1)$, $q >0$, $z<0$ and
$(w,z,p,q) \neq (0,- \frac{1}{3},0,3)$.
\newline
\noindent
The multiplicative loops $Q^{\ast }_{w,z,p,q}$ of the quasifields $Q_{w,z,p,q}$ are not decomposable and quasi-simple.
The left translations of $Q^{\ast }_{w,z,p,q}$ corresponding to the kernel of $Q_{w,z,p,q}$ have the form
$\left( \begin{array}{cc}
r & 0 \\
0 & r \end{array} \right), 0 \neq r \in \mathbb R$, if and only if $w=p=0$, $q=-z=1$.
\end{Prop}
\begin{proof} By Satz 5 in \cite{betten4}, the planes ${\cal A}_{w,z,p,q}$ are determined by the normalized spreads which have the form
{\tiny \begin{equation}
\{ S \} \cup \left\{ \left( \begin{array}{cc}
\alpha  & \frac{-p}{q} \alpha + \frac{p}{q} \left( \alpha + \frac{\beta ^2}{2} \right)^{\frac{3}{2}}+ \frac{(1-q)}{q} \beta 
\left( \alpha + \frac{\beta ^2}{2} \right)- \frac{\beta^3}{3q}  \\
\beta & \frac{-p}{q} \beta + \frac{\beta^2}{2q}+ \left( \alpha + \frac{\beta ^2}{2} \right) \end{array} \right), \beta \in \mathbb R, \alpha \ge - \frac{\beta^2}{2} \right\} \cup \nonumber \end{equation}
\begin{equation}
\left\{ \left( \begin{array}{cc}
\alpha  & \frac{-p}{q} \alpha + \frac{w}{q} \left( -\alpha - \frac{\beta ^2}{2} \right)^{\frac{3}{2}}+ \frac{(z+1)}{q} \beta 
\left( \alpha + \frac{\beta ^2}{2} \right)- \frac{ \beta^3}{3q}  \\
\beta & \frac{-p}{q} \beta + \frac{\beta ^2}{2q} -\frac{z}{q} \left( \alpha + \frac{\beta ^2}{2} \right) \end{array} \right), 
\beta \in \mathbb R, \alpha < - \frac{\beta^2}{2} \right\}. \nonumber \end{equation} }
These spreads coincide with the image of the section $\sigma $ in (\ref{section})  such that the well defined  functions 
$a(r,t)$ and 
$b(r,t)$ are given in the assertion. One gets that $a(r,0)=1$ and $a(r, \pi)= \sqrt{\frac{-q}{z}}$ for all $r>0$. 

For $\beta > 2$ we obtain 
{\tiny \begin{equation} \label{equa} a\left(r, \frac{\pi}{4}+ \pi \right)=\frac{\sqrt{2} \beta }{\sqrt{\frac{\beta^4}{3q}- 
\frac{\beta ^3}{2q}- \frac{w}{q} \beta \left(\beta -\frac{\beta ^2}{2} \right)^{\frac{3}{2}}+ \left(\beta -\frac{\beta ^2}{2} \right) \left(\frac{z+1}{q} \beta ^2 - \frac{\beta z}{q} \right)}}.  \nonumber \end{equation} }
The loop $Q^{\ast }_{w,z,p,q}$ is not decomposable since we have a contradiction to the condition 
$a(r,\frac{\pi }{4}+ k \pi)=a(1,\frac{\pi }{4}) a(r, k \pi)$, $r>0$, $k\in \{0, 1\}$  (cf. Proposition \ref{productloop}).
Hence $Q^{\ast }_{w,z,p,q}$ is quasi-simple (cf. Theorem \ref{quasisimple}). As $a(r,k \pi)=1$ and $b(r,k \pi)=0$, $r>0$, $k \in \{0, 1\}$ holds precisely if 
$w=p=0$, $q=-z=1$ the last assertion follows.  \end{proof}

\begin{Prop} \label{7dimSO2R}
Let $Q$ be a $2$-dimensional quasifield coordinatizing a $4$-dimensional locally compact translation plane ${\cal A}$
such that the translation complement $C$ of the $7$-dimensional  collineation group ${\cal A}$ fixes two distinct lines $\{ S, W \}$ through the origin and leaves on $S$ one or two $1$-dimensional subspaces invariant. Then the multiplicative loop $Q^{\ast }$ can be given by one of the following sets $\Lambda _{Q^{\ast }}$ of the left translations of $Q^{\ast }$ having the form (\ref{lambda}):
\newline
\noindent
a) \begin{equation} \label{a(r,0)}
a(r,0)= r^{\frac{1-w}{1+w}}, \ \ b(r,0)= c \left(r^{\frac{w-1}{w+1}}-r^{\frac{1-w}{1+w}} \right), \nonumber \end{equation}
with $r=s^{\frac{w+1}{2}}$, $s>0$, $t=- \varphi $, where $s$ and $\varphi $ are variables of the spreads (\ref{spreadbetten5satz1}).
The quasifields $Q_{w,c}$ coordinatize a family of planes
${\cal A}_{w,c}$ such that for the parameters $w \neq 1, c$ one has $0 < w$ and $(w-1)^2 c^2 \le 4w$.

\noindent
b) \begin{equation} \label{b(r,0)}
a(r,0)=1, \ \ b(r,0)= \frac{\ln r}{d}, \nonumber \end{equation}
with $r=e^s$, $t=- \varphi $, where $s$ and $\varphi $ are variables of the spreads (\ref{spreadbetten5satz2}). The quasifields $Q_{d}$ coordinatize a one-parameter family of planes
${\cal A}_{d}$ such that $4d^2 \ge 1$.

\noindent
In both cases $Q^{\ast }$ is decomposable and contains the group ${\mathrm{SO}}_2(\mathbb R)$.
\end{Prop}
\begin{proof}
If the group $C$ fixes two $1$-dimensional subspaces of $S$, respectively only one $1$-dimensional subspace of $S$,
then one obtains a family of translation planes corresponding to the normalized spreads
\begin{equation} \label{spreadbetten5satz1} \{ S, W \} \cup \left\{\left( \begin{array}{cc}
\cos \varphi  &  - \sin \varphi \\
\sin \varphi &  \cos \varphi \end{array} \right) \left( \begin{array}{cc}
s & c(s^w-s) \\
0 & s^w \end{array} \right), \ s, \varphi \in \mathbb R, s>0 \right\} \end{equation}
(cf.  \cite{betten5},  Satz \ 1 \ and \ \cite{betten7}, \ p. \ 15), respectively
\begin{equation} \label{spreadbetten5satz2} \{ S, W \} \cup \left\{ \left( \begin{array}{cc}
\cos \varphi  &  - \sin \varphi \\
\sin \varphi &  \cos \varphi \end{array} \right) \left( \begin{array}{cc}
e^s & e^s \frac{s}{d} \\
0 & e^s \end{array} \right), s, \varphi \in \mathbb R \right\},  \end{equation}
(cf. \cite{betten5}, Satz  2  and  \cite{betten7}, p. 15).
In both cases these spreads coincide with the set $\Lambda ={\mathrm{SO}}_2(\mathbb R) {\mathcal K}$ given in (\ref{lambda}) such that the set
${\mathcal K}$ corresponding to the kernel $K_r$ of $Q$ is determined by the functions $a(r,0)$, $b(r,0)$ as in assertion a), respectively  b). \end{proof}

\begin{Rem}
In \cite{betten1} D. Betten constructed $4$-dimensional locally compact non-desarguesian planes ${\cal A}_{f}$ corresponding to continuous, non-linear, strictly monotone functions $f$ defined for $0 \le u \in \mathbb R$ with $f(0)=0$ and
$\lim \limits _{u \to \infty } f(u)= \infty $. The planes ${\cal A}_{f}$ are determined by the normalized spreads
\begin{equation} \left\{ \left( \begin{array}{cc}
u \cos \varphi  & -\frac{f(u) \sin \varphi }{f(1)} \\
u \sin \varphi & \frac{f(u) \cos \varphi}{f(1)} \end{array} \right), u > 0, \ \varphi \in [0, 2 \pi) \ \right\}. \nonumber \end{equation}
These spreads coincide with the set $\Lambda ={\mathrm{SO}}_2(\mathbb R) {\mathcal K}$ given in (\ref{lambda}) such that the set
${\mathcal K}$ corresponding to the kernel $K_r$ of the quasifield $Q_f$ coordinatizing ${\cal A}_{f}$ is determined by the functions
$a(r,0)=\sqrt{\frac{u f(1)}{f(u)}}$, $b(r,0)=0$ with
$r= \sqrt{\frac{u f(u)}{f(1)}}$, $t= - \varphi $, $u \neq 0$.
For $f(u)=f(1) u^w$ these planes are planes in Proposition \ref{7dimSO2R} a) with $c=0$ and
$a(r,0)=r^{\frac{1-w}{1+w}}$. Otherwise the full collineation group of the planes ${\cal A}_{f}$ has dimension $6$.
\end{Rem}

\begin{Prop} \label{prop5}
Let $Q$ be a $2$-dimensional quasifield coordinatizing a $4$-dimensional locally compact translation plane ${\cal A}$
such that the translation complement $C$ of the $7$-dimensional  collineation group of ${\cal A}$ fixes two distinct lines $\{ S, W \}$ through the origin and acts transitively on the spaces $P_S$ and $P_W$ of all $1$-dimensional subspaces of $S$, respectively $W$.  Then the multiplicative loop $Q^{\ast }$ of $Q$ can be given by one of the following sets $\Lambda _{Q^{\ast }}$ of the left translations of
$Q^{\ast }$:
\newline
\noindent
a) $\Lambda _{Q^{\ast }}$ is the range of the section (\ref{section}) with
{\tiny \begin{equation}
a(r,u)=\sqrt{\frac{d D}{d e^{2(qt-ps)}+ d e^{2 q \pi }+ e^{qt-ps+q \pi} \left( 2 d \cos s \cos t+(c^2+ 1+d^2) \sin s \sin t \right)}}, \nonumber  \end{equation}
\begin{equation}
b(r,u)= \frac{e^{2(qt-ps)} \left[(-c^2-1+d^2) \cos t \sin t- c(c^2+1+d^2) \sin^2 t \right]}{\sqrt{d D \left[d(e^{2(qt-ps)}+ e^{2 q \pi})+
e^{qt-ps+q \pi}(2 d \cos s \cos t+(d^2+c^2+1) \sin s \sin t) \right]}} + \nonumber \end{equation}
\begin{equation}
+ \frac{e^{qt-ps+q \pi} \left( \cos s \cos t + d \sin s \sin t+ c \cos s \sin t \right)}{\sqrt{d D \left[d(e^{2(qt-ps)}+ e^{2 q \pi})+ e^{qt-ps+q \pi}(2 d \cos s \cos t+(d^2+c^2+1) \sin s \sin t) \right]}}, \nonumber \end{equation} }
such that
{\tiny \begin{equation} r \cos u= \frac{e^{qt-ps} (\cos{s} \cos{t}+c \sin{t} \cos{s} +d \sin{t} \sin{s})+ e^{q \pi}}{1+e^{q \pi}} a^{-1}(r,u), \nonumber \end{equation}  }
{\tiny \begin{equation} r \sin u= -\frac{e^{qt-ps} (d \cos{s} \sin{t} -\sin{s} \cos{t}- c \sin {s} \sin{t})}{1+e^{q \pi}} a^{-1}(r,u), \nonumber \end{equation}
\begin{equation}  D= e^{2(qt-ps)} \left( (\cos t+c \sin t)^2+d^2 \sin^2 t \right)+ e^{2 q \pi }+2 e^{qt-ps+q \pi}(\cos s \cos t+c \cos s \sin t+d \sin s \sin t). \nonumber \end{equation} }
\newline
\noindent
The quasifields  $Q_{p,q,c,d}$ coordinatize a family of planes $\mathcal{A}_{p,q,c,d}$ such that the parameters
$p,q,c,d$ satisfy the  conditions
{\tiny \begin{equation}
\begin{array}{lcl}
p=q>0 & \hbox{and} & -1 \le d <0, \\
q > 0, p=\frac{k-1}{k+1} q, k=1,2,3, \cdots & \hbox{and} &  d >0,  \end{array} \nonumber \end{equation}
\begin{equation}
-(q+p)^2 A+ (q-p)^2 B -4 A B \ge 0, \ \hbox{where} \ A=\frac{(d-1)^2+c^2}{4d} \ \hbox{and} \ B=\frac{(d+1)^2+c^2}{4d}. \nonumber \end{equation} }
The multiplicative loops $Q^{\ast }$ of the quasifields $Q_{p,q,c,d}$ are not decomposable and quasi-simple.

\smallskip
\noindent
b) $\Lambda _{Q^{\ast }}$ has the form (\ref{diagonal}) with
{\tiny \begin{equation} \label{bjram} a(1,u)=\sqrt{(\cos {nt} + c \sin {nt})^2+ d^2 \sin^2{nt}}, \
b(1,u)= \frac{\sin {nt} \cos{nt}(d^2-1-c^2)-c \sin^2{nt}(d^2+1+c^2)}{d \sqrt{(\cos {nt} + c \sin {nt})^2+ d^2 \sin^2{nt}}} \nonumber \end{equation} }
such that
{\tiny \begin{equation} \label{bjramcos}
r \cos u=\frac{s(\cos{nt} \cos{mt}+c \sin{nt} \cos{mt} +d \sin{nt} \sin{mt})}{\sqrt{(\cos {nt} + c \sin {nt})^2+ d^2 \sin^2{nt}}}, \
r \sin u=\frac{s(d \sin{nt} \cos{mt} -\cos{nt} \sin{mt}- c \sin {nt} \sin{mt})}{\sqrt{(\cos {nt} + c \sin {nt})^2+ d^2 \sin^2{nt}}} \nonumber \end{equation} }
and $s \ge 0$.
\newline
\noindent
The quasifields $Q_{m,n,c,d}$ coordinatize a family of planes $\mathcal{A}_{m,n,c,d}$ such that the parameters
$m,n \in \mathbb Z$, $(m,n)=1$, $c,d \in \mathbb R$ satisfy the  conditions
{\tiny \begin{equation}
\begin{array}{lcl}
m=n=1 & & \hbox{and} \ \ -1 \le d <0 \\
m=1,2,3, \cdots & n=m+1 & \hbox{and} \ \ d >0 \\
m=1,3,5, \cdots & n=m+2 & \hbox{and} \ \ d >0 \end{array} \nonumber \end{equation}
\begin{equation}
(n-m)^2 B \ge (n+m)^2 A, \ \hbox{where} \ A=\frac{(d-1)^2+c^2}{4d} \ \hbox{and} \ B=\frac{(d+1)^2+c^2}{4d}.
  \nonumber \end{equation} }
The loops $Q^{\ast }_{m,n,c,d}$  are split extensions of the normal subgroup
$\widetilde{N^{\ast }} \cong \mathbb R$ corresponding to the connected component of $\left\{ \left( \begin{array}{cc}
u & 0 \\
0 & u \end{array} \right), 0 \neq u \in \mathbb R \right\}$ with a subloop homeomorphic to the $1$-sphere.
\end{Prop}
\begin{proof}
If the translation complement $C$ acts transitively on the product space $P_S \times P_w$, then there is a family of translation planes corresponding to the normalized spreads
\begin{equation} \label{spread1uj}
\{S, W \} \cup \left\{ \left( \begin{array}{cc}
\frac{\alpha (s,t)+ e^{q \pi}}{1+ e^{q \pi}} & \frac{\gamma (s,t)-c \alpha (s,t)}{d(1+ e^{q \pi})} \\
\frac{\beta (s,t)}{1+ e^{q \pi}} & \frac{\delta (s,t)-c \beta (s,t)+ d e^{q \pi}}{d(1+ e^{q \pi})} \end{array} \right), s,t \in \mathbb R  \right\} \nonumber
\end{equation}
such that
$\alpha (s,t)= e^{qt-ps} (\cos{s} \cos{t}+c \sin{t} \cos{s} +d \sin{t} \sin{s})$,
\newline
\noindent
$\beta (s,t)= e^{qt-ps} (d \cos{s} \sin{t} -\sin{s} \cos{t}- c \sin {s} \sin{t})$,
\newline
\noindent
$\gamma(s,t)= e^{qt-ps} (d \cos{t} \sin{s} -\sin{t} \cos{s}+ c \cos {t} \cos{s})$,
\newline
\noindent
$\delta(s,t)= e^{qt-ps} (d \cos{t} \cos{s} +\sin{t} \sin{s}- c \cos {t} \sin{s})$  (cf. \cite{betten5}, \ \hbox{Satz 3}, pp. 135-136). These spreads coincide with the image of the section $\sigma $ in (\ref{section})
with the well defined  functions $a(r,u)$ and $b(r,u)$ as in assertion a).
For $s=0$ we get a contradiction to the condition $a(r_j,u_j)= a(r_j,0) a(1, u_j)$ which must hold for decomposable $Q^{\ast }$.
It follows that $Q^{\ast }$ is not decomposable and hence quasi-simple (cf. Theorem \ref{quasisimple}).

\noindent
If the translation complement $C$ does not act transitively on the product space $P_S \times P_W$, then there is a family of translation planes which correspond to the normalized spreads
\begin{equation} \label{spread1r}
\{S, W \} \cup \left\{ \left( \begin{array}{cc}
s & 0 \\
0 & s \end{array} \right) \left( \begin{array}{cc}
a_{11}(t) & -\frac{c}{d} a_{11}(t)+ \frac{1}{d} a_{21}(t) \\
a_{12}(t) & -\frac{c}{d} a_{12}(t)+ \frac{1}{d} a_{22}(t) \end{array} \right), s \ge 0, t \in \mathbb R \right\} \nonumber \end{equation}
with
$a_{11}(t)= \cos{nt} \cos{mt}+c \sin{nt} \cos{mt} +d \sin{nt} \sin{mt}$,
\newline
\noindent
$a_{12}(t)= d \sin{nt} \cos{mt} -\cos{nt} \sin{mt}- c \sin {nt} \sin{mt}$,
\newline
\noindent
$a_{21}(t)= d \cos{nt} \sin{mt} -\sin{nt} \cos{mt}+ c \cos {nt} \sin{mt}$,
\newline
\noindent
$a_{22}(t)= d \cos{nt} \cos{mt} +\sin{nt} \sin{mt}- c \cos {nt} \sin{mt}$  (cf. \cite{betten5}, \ \hbox{Satz 4}, pp. 142-144).
These spreads coincide with the set $\Lambda_{Q^{\ast}}$ in (\ref{diagonal}) such that the periodic functions $a(1,t)$ and $b(1,t)$ are given in assertion b). As in the proof of Proposition \ref{prop1} a) it follows that
the loop $Q^{\ast }_{m,n,c,d}$  is a split extension as in the assertion.
\end{proof}

\begin{Coro}
 Let $\mathcal{A}$ be a  $4$-dimensional locally compact non-desarguesian topological plane which admits an at least $7$-dimensional collineation group $\Gamma$. If the quasifield $Q$ coordinatizing  $\mathcal{A}$  is constructed with respect to two lines such that their intersection points with the line at infinity are contained in the $1$-dimensional orbit of $\Gamma $ or contain the set of the fixed points of $\Gamma$, then for the multiplicative loop $Q^\ast$ of  $Q$  one of the following holds:
\newline
\noindent
a) $Q^{\ast }$ is quasi-simple and not decomposable. Such quasifields $Q$ are described by Propositions \ref{prop1} b), \ref{prop1} c), \ref{prop2}), \ref{prop3}), \ref{prop4}) and in Proposition \ref{prop5} a).
\newline
\noindent
b) $Q^{\ast }$ is quasi-simple but decomposable and it is a product $SO_2(\mathbb R) B$, where $B$ is a $1$-dimensional loop homeomorphic to
$\mathbb R$. The quasifields $Q$ of this type are described in Proposition \ref{7dimSO2R}.
\newline
\noindent
c) $Q^{\ast }$ is a split extension of the group $\widetilde{N^{\ast }} \cong \mathbb R$ with a loop homeomorphic to the 
$1$-sphere.
The quasifields of this type are described in Propositions \ref{prop1} a) and \ref{prop5} b).
\end{Coro}
\begin{proof}
A locally compact topological quasifield coordinatizing the translation plane $\mathcal{A}$ and constructed with respect to two lines satifying  the assumptions is
isotopic to a quasifield given in Betten's classification (cf. \cite{grundhoefer}, p. 321, \cite{betten} Satz 5). For isotopic loops $Q_1^{\ast }$ and $Q_2^{\ast }$ the following holds:
The group generated by their left translations, every subgroup and all nuclei of them are isomorphic (cf. \cite{loops}, Lemmata 1.9, 1.10, p. 20). From these facts we get: If $Q_1$ is quasisimple and not decomposable, then also $Q_2$ is quasisimple and not decomposable.
If $Q_1$ contains the subgroup $SO_2(\mathbb R)$, then also $Q_2$ contains the group $SO_2(\mathbb R)$. If $Q_1$ is a split extension of $\widetilde{N^{\ast }}$ with a $1$-dimensional compact loop, then the same holds for $Q_2$.
\end{proof}

\medskip
\noindent
{\bf Acknowledgement.}
We are particularly grateful to the referee for his suggestions. This paper was supported by the J\'anos Bolyai Research Fellowship and
the European Union's Seventh Framework Programme (FP7/2007-2013) under grant agreements no. 317721, no. 318202.


\begin{thebibliography}{37}

\bibitem{andre} Andr\'e, J. \"Uber nicht-Desarguessche Ebenen mit transi\-ti\-ver Trans\-lationsgruppe. \textit{Math. Z.} 60:156-186.

\bibitem{betten1} Betten, D. Nicht-desarguessche $4$-dimensionale Ebenen. \textit{Arch. Math.} 21:100-102.

\bibitem{betten} Betten, D. $4$-dimensionale Translationsebenen. \textit{Math. Z.} 128:129-151.

\bibitem{betten2} Betten, D. $4$-dimensionale Transla\-ti\-ons\-ebe\-nen mit $8$-dimensionaler Kol\-li\-nea\-tions\-gruppe. \textit{Geom. Ded.} 2:327-339.

\bibitem{betten3} Betten, D. $4$-dimensionale Translationsebenen mit irreduzibler Kol\-li\-nea\-tions\-gruppe. \textit{Arch. Math.} 24:552-560. \textit{Math. Z.} 132:249-259.

\bibitem{betten4} Betten, D. 4-dimensionale Translationsebenen mit genau einer Fixrichtung. \textit{Deom. Ded.} 3:405-440.

\bibitem{betten5} Betten, D. 4-dimensionale Translationsebenen mit 7-dimensionaler Kol\-li\-neationsgruppe. \textit{J. Reine Angew. Math.} 285:126-148.

\bibitem{betten6} Betten, D. 4-dimensionale Translationsebenen mit kommutativer Standgruppe. \textit{Math. Z.} 154:125-141.

\bibitem{betten7} Betten, D. (1997) On the classification of 4-Dimensional Flexible Projective Planes. In: Johnson, N. L., ed. {\textit Mostly finite geometries.} Decker, pp. 9-33.

\bibitem{figulastrambach}  Figula, \'A.,  Strambach, K. Loops on spheres having a compact-free inner mapping group. \textit{Monatsh. Math.} 156:123-140.

\bibitem{grundhoefer}  Grundhöfer, T.,  Salzmann, H. H. (1990). Locally compact double loops and ternary fields. In: Chein, O., et al., ed. \textit{Quasigroups and Loops: Theory and Applications.}  Heldermann, pp. 313-355.


\bibitem{hughes}  Hughes, D. R., Piper, F. C. (1973). Projective Planes. Springer.


\bibitem{knarr} Knarr,  N. (1995). Translation Planes. Springer.

\bibitem{loops} Nagy,  P. T., Strambach, K. (2002). Loops in group theory and Lie theory. Walter de Gruyter.


\bibitem{ortleb} Ortleb,  S.  A new family of locally compact $4$-dimensional translation planes admitting a $7$-dimensional collineation group. \textit{Adv. Geom.} 9:1-12.


\bibitem{pickert}  Pickert, G. (1955) Projektive Ebenen. Springer.


\bibitem{strambachplaumann} Plaumann, P.,  Strambach, K. Zweidimensionale Quasialgebren mit Nullteilern. \textit{Aequationes Math.} 15:249-264.


\bibitem{salzmann} Salzmann, H., Betten, D.,  Grundh\"ofer, T., H\"ahl, H.,  L\"owen, R., Stroppel, M. (1995). Compact projective planes.
Walter de Gruyter.


\bibitem{walter} Walter,  W. (1970).  Differential and Integral Inequalities. Springer.

\end{thebibliography}
\end{document}